\documentclass[11pt, letterpaper]{amsart} 
\usepackage[margin=1in]{geometry}
\usepackage{setspace}
\usepackage{graphicx}	
\usepackage{subcaption}
\usepackage{float}
\usepackage[toc]{appendix}
\usepackage{bm}
\usepackage[backref=page]{hyperref}
\usepackage{amsfonts}
\usepackage{dsfont}

\usepackage{enumitem}
\usepackage{lineno}

\usepackage{amssymb}
\usepackage{amsmath, amsthm, amssymb, mathrsfs}

\newtheorem{thm}{Theorem}[section]

\newtheorem{prop}[thm]{Proposition}

\newtheorem{lemma}[thm]{Lemma}
\newtheorem{cor}[thm]{Corollary}
\newtheorem{prob}[thm]{Problem}
\newtheorem{observation}[thm]{Observation}

\theoremstyle{definition}
\newtheorem{defi}[thm]{Definition}

\def\cover{{\bf COVER}}
\def\factor{{\bf FACTOR}}
\def\trans{{\bf TRANS}}

\usepackage{tikz}
\tikzstyle{aNode} = [circle, fill = black]
\tikzstyle{bNode} = [circle,draw = black, thick]

\newcommand{\pcherry}[1]{%
\begin{tikzpicture}[inner sep = 0.7pt, #1]%
\node (1) at (0,-2) [aNode]{};
\node (3) at (1.5,-2) [aNode]{};
\node (2) at (0.75,-1) [aNode]{};
\draw  (1) -- (2);
\draw  (2) -- (3);
\end{tikzpicture}%
}

\newcommand{\ppoints}[1]{%
\begin{tikzpicture}[inner sep = 0.7pt, #1]%
\node (1) at (0,-2) [aNode]{};
\node (3) at (1.5,-2) [aNode]{};
\node (2) at (0.75,-1) [aNode]{};
\end{tikzpicture}%
}

\newcommand{\pedge}[1]{%
\begin{tikzpicture}[inner sep = 0.7pt, #1]%
\node (1) at (0,-2) [aNode]{};
\node (3) at (1.5,-2) [aNode]{};
\node (2) at (0.75,-1) [aNode]{};
\draw  (1) -- (3);
\end{tikzpicture}%
}

\def\cherry{\pcherry{scale=0.11}}

\def\points{\ppoints{scale=0.11}}
\def\edge{\pedge{scale=0.14}}

\newcounter{casenum}

\newcommand*{\rom}[1]{\expandafter{\romannumeral #1\relax}}

\begin{document}
\title[Tiling multipartite hypergraphs in Quasi-random Hypergraphs]{Tiling multipartite hypergraphs in Quasi-random Hypergraphs}
\date{\today}
\author{Laihao Ding}
\author{Jie Han}
\author{Shumin Sun}
\author{Guanghui Wang}
\author{Wenling Zhou}

\address
{School of Mathematics and Statistics, Central China Normal University,
Wuhan, China.}
\thanks{LHD is supported by the National Natural Science Foundation of China (11901226).
SMS, GHW and WLZ are partially supported by the Natural Science Foundation of China  (11871311, 12231018) and Shandong University multidisciplinary research and innovation team of young scholars.}
\email[Laihao Ding]{dinglaihao@mail.ccnu.edu.cn}
\address
{School of Mathematics and Statistics, Center for Applied Mathematics, Beijing Institute of Technology, Beijing, China.}
\email[Jie Han]{han.jie@bit.edu.cn}
\address
{School of Mathematics, Shandong University, Jinan, China.}
\email[Shumin Sun]{sunshumin987@163.com}
\email[Guanghui Wang]{ghwang@sdu.edu.cn}
\address
{School of Mathematics, Shandong University,
	Jinan, China, and  Laboratoire Interdisciplinaire des Sciences du Num\'{e}rique, CNRS-Universit\'{e} Paris-Saclay, Orsay, France.}
\email[Wenling Zhou]{gracezhou@mail.sdu.edu.cn}
\date{}

\begin{abstract}
Given $k\ge 2$ and two $k$-graphs ($k$-uniform hypergraphs) $F$ and $H$, an \emph{$F$-factor} in $H$ is a set of vertex disjoint copies of $F$ that together covers the vertex set of $H$. Lenz and Mubayi
studied the $F$-factor problems in quasi-random $k$-graphs with minimum degree $\Omega(n^{k-1})$. In particular, they constructed a sequence of $1/8$-dense quasi-random $3$-graphs $H(n)$ with minimum degree $\Omega(n^2)$ and minimum codegree $\Omega(n)$ but with no $K_{2,2,2}$-factor. 
We prove that if $p>1/8$ and $F$ is a $3$-partite $3$-graph with $f$ vertices, then for sufficiently large $n$, all $p$-dense quasi-random $3$-graphs of order $n$ with minimum codegree $\Omega(n)$ and $f\mid n$ have $F$-factors. 
That is, $1/8$ is the density threshold for ensuring all $3$-partite $3$-graphs $F$-factors  in quasi-random $3$-graphs given a minimum codegree condition $\Omega(n)$. Moreover, we show that one can not replace the minimum codegree condition by a minimum vertex degree condition. In fact, we find that for any $p\in(0,1)$ and $n\ge n_0$, there exist $p$-dense quasi-random $3$-graphs of order $n$ with minimum degree $\Omega (n^2)$ having no $K_{2,2,2}$-factor. In particular, we study the optimal  density threshold of $F$-factors for each $3$-partite $3$-graph $F$ in quasi-random $3$-graphs given a minimum codegree condition $\Omega(n)$.

In addition, we also  study $F$-factor problems for 
$k$-partite $k$-graphs $F$ with stronger quasi-random assumption and minimum degree $\Omega(n^{k-1})$.
\end{abstract}

\maketitle

\section{Introduction}
For a positive integer $\ell$, we denote by $[\ell]$ the set $\{1,\dots,\ell\}$.
Given $k\geq 2$, for a finite set $V$, we use  $[V]^k$ to denote the collection of all subsets of $V$ of size $k$.
A \textit{$k$-uniform hypergraph} $H$ ($k$-graph for short) is a pair $H=(V(H),E(H))$ where $V(H)$ is a finite set of \textit{vertices} and $E(H)\subseteq [V(H)]^k$ is a set of \textit{edges}. As usual $2$-graphs are simply called graphs.
We simply write $v(H)=|V(H)|$ and $e(H)=|E(H)|$. Let $\partial H:=\{S\in [V(H)]^{k-1}: \exists \  e\in E(H), S\subseteq e\}$ be the \emph{shadow} of $H$.
Given a $k$-graph $F$, a \textit{$k$-partite realisation} of $F$ is a partition of $V(F)$ into vertex classes $V_1, \dots, V_k$ so that for any $e\in E(F)$ and $i\in [k]$ we have $|e\cap V_i| =1$. Clearly, we must have $|V_i| \ge 1$ for every $i\in [k]$. In particular, we say that $F$ is \textit{$k$-partite} if it admits a $k$-partite realisation.
Given a $k$-graph $H=(V, E)$ and a set $S$ of $s$ vertices in $V$ with $1\le s\le k-1$, let $\deg_H(S)$ (or $\deg(S)$) denote the number of edges of $H$ containing $S$ as a subset. Let $N_H(S)$ (or $N(S)$) denote the neighbor set of $S$, i.e., $N_H(S)=\{ S' \in [V]^{k-s}\ : S'\cup S\in E\}$.
The \textit{minimum $s$-degree} $\delta_s(H)$ is the minimum of $\deg(S)$ over all $s$-subsets $S$ of $V$. In particular, we call the minimum $1$-degree and the minimum $(k-1)$-degree of $H$ the \textit{minimum degree} and \textit{minimum codegree} of $H$ respectively. If $S=\{v\}$ is a singleton, then we simply write $\deg(v)$ and $N(v)$  instead.

In this paper, we focus on a classical extremal problem, the perfect tiling problem.
Given two $k$-graphs $F$ and $H$, a \textit{perfect $F$-tiling (or $F$-factor)} in $H$ is a set of vertex disjoint copies of $F$ that together covers the vertex set of $H$.
Furthermore, if $H$ does not contain copies of $F$, we say $H$ is \textit{$F$-free}.
The study of perfect tilings in graph theory has a long and profound history, ranging from classical results back to Corradi--Hajnal~\cite{Corr} and Hajnal--Szemer\'edi~\cite{zbMATH03344609} on $K_k$-factors to the celebrated result of Johansson--Kahn--Vu~\cite{Johansson2008Factors} on perfect tilings in random graphs.
We note that the $F$-factor problem for quasi-random graphs with positive density and minimum degree $\Omega(n)$ has been solved implicitly by Koml\'{o}s--S\'{a}rk\"{o}zy--Szemer\'{e}di~\cite{Koml1997Blow} in the course of developing the famous Blow-up Lemma.

\subsection{Quasi-random graphs and hypergraphs}
The study of quasi-random graphs was launched in the late 1980s by Chung, Graham and Wilson~\cite{Chung-quasi-random}.
These are constant density graphs which behave like random graphs.
There is a list of properties that force a graph to be quasi-random and it was shown that these properties are all equivalent (up to the dependency of constants).
The investigation of quasi-random $k$-graphs was started by Chung and Graham~\cite{chung1990quasi} and is widely popular   \cite{quasihyper,Ch90,Ch91,Chung10,wquasi,G06,G07,KNRS,KRS02,LenzMubayi_eig,lenz2015poset,RSk04,Towsner}.
Unlike in graphs, there are several non-equivalent notions of quai-randomness in $k$-graphs for $k\ge 3$ (see~\cite{Reiher_2017}).
In particular, the subgraph containment problem for quasi-random $k$-graphs, $k\geq3$, is quite different from the case $k=2$ and has been an interesting topic over decades.
Indeed, for $k\geq 3$, R\"odl noted that by a construction from~\cite{ErdosHajnal}, quasi-random $k$-graphs may not contain a single copy of, say, a $(k+1)$-clique.
In fact below we will consider two different notions of quasi-random $k$-graphs,
which have been studied by Reiher, R\"odl and Schacht~\cite{MR3548293}.
The first and weakest concept that we consider here is $(p, \mu, \points)$-quasi-randomness.

\begin{defi}
Given integers $n\geq k\ge 2$, let real numbers $p \in [0, 1]$, $\mu >0$, and $H=(V,E)$ be a $k$-graph with $n$ vertices. We say that $H$ is ($p,\mu,\points$)\emph{-dense} if for all $X_1,\dots,X_k\subseteq V$,
\begin{equation}\label{eq:count}
e_{\points}(X_1,\dots,X_k)\geq p|X_1|\cdots|X_k|-\mu n^k,
\end{equation}
where $e_{\points}(X_1,\dots,X_k)$ is the number of $(x_1,\dots,x_k)\in X_1\times \cdots \times X_k$ such that $\{x_1,\dots,x_k\}\in E$.
\end{defi}

Next we come to a stronger concept of quasi-random hypergraphs where rather than ``$k$ sets of vertices" we consider ``a set of vertices and a set of $(k-1)$-elements".

\begin{defi}
Given integers $n\geq k\ge 2$, let real numbers $p \in [0, 1]$, $\mu >0$, and $H=(V,E)$ be a $k$-graph with $n$ vertices. We say that $H$ is ($p,\mu,\edge$)\emph{-dense} if for all $X\subseteq V$ and  $Y\subseteq V^{k-1}$,
\begin{equation}\label{eq:count1}
e_{\edge}(X,Y)\geq p|X||Y|-\mu n^k,
\end{equation}
where $e_{\edge}(X,Y)$ is the number of pairs $(x,(y_1, \dots, y_{k-1}))\in X\times Y$ such that $\{x,y_1, \dots, y_{k-1}\}\in E$.
\end{defi}

Throughout the paper, we say that $H$ is an \emph{$(n,p,\mu, \mathscr{S})$ $k$-graph} if $H$ has $n$ vertices and is ($p,\mu,\mathscr{S}$)-dense for a symbol $\mathscr{S}\in \{\points, \edge\}$.
Given $p\in
[0,1]$ and $\mathscr{S}\in \{\points, \edge\}$, we sometimes write a $k$-graph is \emph{$p$-dense $\mathscr{S}$-quasi-random} 
 to mean that it is ($p,\mu,\mathscr{S}$)-dense for some small $\mu$.
Let $\mathcal{C}(n,p,\mu, \mathscr{S})$ denote the class of all ($n,p,\mu, \mathscr{S}$) $k$-graphs and one can observe that
\[
\mathcal{C}(n,p,\mu, \edge)\subseteq \mathcal{C}(n,p,\mu, \points)
\]
holds for all $p\in [0,1]$ and $\mu>0$.
Now for $\mathscr{S}\in \{\points, \edge\}$, we give a formal definition of Tur\'an densities for $\mathscr{S}$-quasi-random $k$-graphs, which is an extension of the definition of Tur\'an densities for quasi-random hypergraphs in~\cite{MR3548293}. Given a $k$-graph $F$, we define:
\begin{equation*}
\begin{split}
\label{dot-turan-dense}
\pi_{\mathscr{S}}(F) = \sup \{ p\in [0,1] & : \text{for\ every\ } \mu>0 \ \text{and\ } n_0\in \mathbb{N},\ \text{there\ exists\ an\ } F \text{-free} \\
&\quad (n, p,\mu,\mathscr{S})\ k \text{-graph\ } H\ \text{with\ } n\geq n_0 \}.
\end{split}
\end{equation*}


\subsection{$F$-factors in quasi-random $k$-graphs}
Lenz and Mubayi~\cite{Lenz2016Perfect} were the first to study the $F$-factor problems in quasi-random $k$-graphs, for $k\ge 3$.
For $\mathscr{S}\in \{\points, \edge\}$, $k\ge 2$ and $1\leq s \leq k-1$, define $\factor_k^{\mathscr{S} ,s}$ as the collection of $k$-graphs $F$ such that for all $0< p,\alpha<1$, there is some $n_0\in \mathbb{N}$ and $\mu > 0$ so that if $H$ is an $(n,p,\mu ,\mathscr{S})$ $k$-graph with $\delta _s(H)\geq \alpha n^{k-s}$, $n \geq n_0$ and $v(F)\mid n$, then $H$ has an $F$-factor.
For $k=2$, by the aforementioned result of~\cite{Koml1997Blow}, all graphs belong to $\factor_2^{\points ,1}$. Thus, Lenz and Mubayi raised the following natural question.

\begin{prob}{\rm(\cite[Problem 2]{Lenz2016Perfect}).}\label{prob}
For $k\ge 3$, which $k$-graphs $F$ belong to $\factor_k^{\points ,1}$?
\end{prob}

In~\cite{Lenz2016Perfect},
Lenz and Mubayi proved that all linear $k$-graphs belong to $\factor_k^{\points ,1}$ and provided a family of $3$-graphs not in $\factor_k^{\points ,1}$. 
More precisely, by a parity-based construction~\cite[Theorem 5]{Lenz2016Perfect}, they constructed an $(n,\frac{1}{8},\mu, \points )$ 3-graph $H^*$ with $\delta _1(H^*)\geq (\frac{1}{8}-\mu )\binom{n}{2}$, $\delta _2(H^*)\geq (\frac{1}{8}-\mu ){n}$ such that $H^*$ has no $F$-factor for those $F$ which have an even number of vertices and admits a partition of $V(F)$ into pairs such that each pair of vertices has a common member in their neighbor sets. 
Let $K_{2,2,2}$ denote the complete 3-partite 3-graph with parts of size two. It is easy to check that $K_{2,2,2}$ satisfies this property.
In our work~\cite{DHSWZ}, we completely solved Problem~\ref{prob} for $k=3$. Moreover, we also characterised all $k$-partite $k$-graphs $F$ with $F\in \factor^{\points,1}_k$.

\begin{thm}{\rm({\cite[Theorem 1.4]{DHSWZ}}).}\label{k-partite}
For $k \geq 3$, a $k$-partite $k$-graph $F\in \factor^{\points,1}_k$ if and only if there exists $v^*\in V(F)$ such that $|e\cap e'|\leq 1$ for any two edges $e,e'\in E(F)$ with $v^*\in e$ and $v^*\notin e'$.
\end{thm}

It is natural to consider the following problem.

\begin{prob}\label{partiteprob}
Given a $k$-partite $k$-graph $F$, what conditions of quasi-randomness and degree for the host $k$-graph $H$ are sufficient to ensure it contains an $F$-factor?
\end{prob}

By the construction of Lenz and Mubayi, we naturally want to know, does $p>1/8$ guarantee that every $p$-dense $\points$-quasi-random 3-graph $H$ with minimum degree $\Omega(n^2)$ has a $K_{2,2,2}$-factor? Unfortunately, this appears to be false by the following result.

\begin{thm}\label{construction}
Let $F$ be a $3$-graph satisfying that for each $v\in V(F)$ there exists a vertex $u$ such that $N(v)\cap N(u)\neq \emptyset$. Then for any $p\in(0,1)$ and $\mu>0$, there exists an $n_0$ and $\alpha>0$ such that for all $n\ge n_0$, there exists an $(n,p,\mu, \points)$ $3$-graph $H$ with $\delta_1(H)\ge \alpha n^{2}$ such that $H$ has no $F$-factor.
\end{thm}

$K_{2,2,2}$ is an example of a $3$-graph $F$ that satisfies the conditions of Theorem~\ref{construction}. For $p$ close to $1$ this shows that $\points$-quasi-randomness and minimum degree conditions of the host $3$-graph $H$ are not sufficient to ensure it contains  $K_{2,2,2}$-factors.
However, if we replace the minimum degree by the minimum codegree condition, then we can prove the following result.

\begin{thm}\label{1/8-dot-codegree}
Given $0<\varepsilon,\alpha <1$, and a $3$-partite $3$-graph $F$ of order $f$,
there exists an $n_0$ and $\mu>0$ such that the following holds for $n\geq n_0$. Let $H$ be an $(n,\frac{1}{8}+\varepsilon,\mu ,\points )$ 3-graph with $\delta_2(H)\geq \alpha n$ and $n\in f\mathbb{N}$. Then $H$ has an $F$-factor.
\end{thm}

As mentioned above, Lenz and Mubayi (see~\cite[Section 4]{Lenz2016Perfect}) constructed an $(n,\frac{1}{8},\mu, \points )$ 3-graph $H^*$ with  $\delta _2(H^*)\geq (\frac{1}{8}-\mu ){n}$ and $H^*$ has no $K_{2,2,2}$-factor. 
That is, $1/8$ is the density threshold for ensuring  all $3$-partite $3$-graphs $F$-factors in $\points$-quasi-random $3$-graphs given a minimum codegree condition $\Omega(n)$.

\subsection{Perfect tilings of $k$-graphs in quasi-random $k$-graphs with positive density}
Now we turn to $p$-dense $k$-graphs when $p>0$ is arbitrarily small, the original motivation of~\cite{Lenz2016Perfect}.
Note that (see~\cite{Han2017Minimum}) a trivial necessary condition for a $k$-graph $H$ to contain an $F$-factor is that every vertex of $H$ is covered by at least one copy of $F$.
This can be visualized as a set of (not necessarily vertex disjoint or distinct) copies of $F$ whose union covers the vertex set of the host graph, so also called an \emph{$F$-cover}. For the $F$-factor problems, the cover property plays an essential role.
Indeed, for $\mathscr{S} \in \{\points, \edge \}$ and $1\le s\le k-1$,  we define  $\cover_{k}^{\mathscr{S}, s}$ as the collection of $k$-graphs $F$ such that for all $0< p,\alpha<1$, there is some $n_0\in \mathbb{N}$ and $\mu > 0$  so that if $H$ is an $(n,p,\mu ,\mathscr{S})$ $k$-graph with $\delta _s(H)\geq \alpha n^{k-s}$ and $n \geq n_0$, then each vertex $v$ in $V(H)$ is contained in a copy of $F$, namely, $H$ has an $F$-cover. 
However, in our proofs, we need a stronger cover property.
More precisely, let $\cover_{k,>}^{\mathscr{S}, s}$ be a collection of $k$-graphs such that a $k$-graph $F$ is in it if for all $0< p,\alpha<1$, there is some $n_0\in \mathbb{N}$ and $\gamma, \mu > 0$  so that if $H$ is an $(n,p,\mu ,\mathscr{S})$ $k$-graph with $\delta _s(H)\geq \alpha n^{k-s}$ and $n \geq n_0$, then each vertex $v$ in $V(H)$ is contained in  at least $\gamma n^{k-1}$ copies of $F$.
We proved the following equivalence relation in~\cite{DHSWZ}, which reduced the problem of $F$-factors to the study of ``covering every vertex of the host $k$-graph by a copy of $F$''.

\begin{thm}{\rm({\cite[Theorem 1.3]{DHSWZ}}).}\label{factor-cover1}
For $k \geq 2$, $\factor^{\points,1}_k=\cover^{\points,1}_{k}$.
\end{thm}

Actually, the proof of Theorem~\ref{factor-cover1} works for other (stronger) quasi-randomness notions (see {\cite[Section 6]{DHSWZ}}).

\begin{thm}\label{factor-cover2}
For any $k \geq 2$, $\factor^{\edge,1}_k=\cover^{\edge,1}_{k}$.
\end{thm}

As an application of Theorem~\ref{factor-cover2}, we will prove the following result.
\begin{thm}\label{dot-edge-partite}
For any $k$-partite $k$-graph $F$, $F\in \factor_k^{\edge,1}$.
\end{thm}

Recalling that the construction of Lenz and Mubayi~\cite{Lenz2016Perfect} showed that $K_{2,2,2}\notin \factor_3^{\points,2}$. 
Since $\factor_3^{\points,1}$ is completely characterised, a natural next step to consider is $\factor_3^{\points,2}$. 
Let us recall some definitions introduced in our work~\cite{DHSWZ}.
Given a $k$-graph $F$ with $k\ge 2$, for a bipartition $\mathcal{P}:=\{V_1,V_2\}$ of $V(F)$ and $e\in E(F)$, we write $\textbf{i} _{\mathcal{P},F}:=(|V_1|, |V_2|)$ and $\textbf{i} _{\mathcal{P},F}(e)=(|e\cap V_1|, |e\cap V_2|)$. 
Given an integer $2\le s\le k-1$, we say that a bipartition $\mathcal{P}$ of $V(F)$ is \emph{$s$-shadow disjoint} if every $e, e'\in E(F)$ with $\textbf{i} _{\mathcal{P},F}(e)\neq \textbf{i} _{\mathcal{P},F}(e')$ satisfy that $|e\cap e'|<s$.
Let $L^{s}_{F}$ be the lattice generated by all $\textbf{i} _{\mathcal{P},F}$ such that $\mathcal{P}$ is $s$-shadow disjoint and we denote $\trans^{s}_k$ as the collection of $k$-graphs $F$ with $(1,-1)\in L^{s}_F$. 
Inspired by Theorem~\ref{factor-cover2}, we prove the following result.

\begin{thm}\label{factor3-2}
 $\cover_{3,>}^{\points ,2}\cap \trans^2_3\subseteq \factor_3^{\points ,2}$.
\end{thm}

In \cite[Section 6]{DHSWZ}, we proved the following result using a random construction.

\begin{observation}{\rm(\cite[Observation 6.4]{DHSWZ}).}\label{construction-trans}
	For $2\le s\le k-1$, $\factor_k^{\points, s} \subseteq \trans^{s}_k$.
\end{observation}

Therefore, by Observation~\ref{construction-trans}, we actually have
\[
\cover_{3,>}^{\points ,2}\cap \trans^2_3\subseteq \factor_3^{\points ,2} \subseteq \cover_{3}^{\points ,2}\cap \trans^2_3.
\]
However, at this moment we are unable to prove $\cover_{3,>}^{\points ,2}= \cover_{3}^{\points ,2}$.
Again, the problem is easier for $3$-partite 3-graphs, in which case we know everything, including the optimal densities.

For 3-graphs $F$, let $p_F$ be the minimum of $p$ such that for every $0<\varepsilon,\alpha <1$, there exists an $n_0$ and $\mu>0$ such that  every $(n, p+\varepsilon,\mu ,\points )$ 3-graph $H$ with $\delta_2(H)\geq \alpha n$, $n\ge n_0$ and $v(F)\mid n$,  satisfies that $H$ has an $F$-factor. 
Then we have the following result.

\begin{thm}\label{3-partite3-graph}
For any $3$-partite $3$-graph $F$, $p_F=0$ if $F\in  \trans^2_3$, and $p_F=1/8$ otherwise. 
\end{thm}

We are unable to extend Theorem~\ref{factor3-2} or Theorem~\ref{3-partite3-graph} to $k\ge 3$, mainly because of our limited understanding of $\trans^s_k$.
However, the case of complete $k$-partite $k$-graphs is easier.
That is, for $k\ge 3$, 
we give a characterization of all complete $k$-partite $k$-graphs in $\factor_k^{\points,k-1}$ and prove that a class of $k$-partite $k$-graphs belongs to $\factor_k^{\points,k-1}$. 
For this, let us recall some definitions introduced by Mycroft~\cite{Mycroft-par}.
Given a $k$-partite $k$-graph $F$, we define 
\[\mathcal{S}(F):=\bigcup_{\chi} \{|V_1|, \dots, |V_k|\},\]
where in each case the union is taken over all $k$-partite realisations $\chi$ of $F$ into vertex classes $V_1, \dots, V_k$ of $V(F)$.
The greatest common divisor of $\mathcal{S}(F)$, is denoted by $\gcd(\mathcal{S}(F))$.
For example, let $K_{a_1,\dots ,a_k}$ denote the complete $k$-partite $k$-graph with parts of size $a_1,\dots ,a_k$. Observe that $K_{a_1,\dots ,a_k}$ has only one $k$-partite realisation up to permutations of the vertex classes $V_1, \dots, V_k$, so $\gcd(\mathcal{S}(K_{a_1,\dots ,a_k}))=\gcd(a_1,\dots ,a_k)$.
For all $k$-partite $k$-graphs $F$ with $\gcd(\mathcal{S}(F))=1$, we have the following result.

\begin{thm}\label{gcd}
If a $k$-partite $k$-graph $F$ satisfies $\gcd(\mathcal{S}(F))=1$, then $F\in \factor_k^{\points,k-1}$.
Moreover, if $F$ is a complete $k$-partite $k$-graph, then $F\in \factor_k^{\points,k-1}$ if and only if $\gcd(\mathcal{S}(F))=1$.
\end{thm}

\noindent {\bf Remark:} Note that for a non-complete $k$-partite $k$-graph $F$, $F\in \factor_k^{\points,k-1}$ does not imply  $\gcd(\mathcal{S}(F))=1$ in general. For example, consider $M_2$ as a matching of size $2$. By Theorem~\ref{factor-cover1},  $M_2\in \factor_k^{\points,1}\subseteq  \factor_k^{\points,k-1}$ but $\gcd(\mathcal{S}(M_2))=2$.

\subsection{Related work.} 
As mentioned above, minimum degree conditions forcing perfect $F$-tilings in graphs have a long history and have been well studied. 
In the past decade there has been substantial interest in extending these results to $k$-graphs (see a recent survey~\cite{Zhao-survey}).
However, only a handful of optimal minimum degree thresholds are known. There are also several results related to our setting, namely, the recent results on locally dense (hyper)graphs. These are notions posed by Erd\H{o}s and  S\'os~\cite{E-S1982} as strengthenings of \emph{having sublinear independence number}. Namely, a $k$-graph $H$ on $n$ vertices is \emph{locally dense} if it satisfies that for every subset $X\subseteq V(H)$, $H[X]$ contains at least $p|X|^k-\mu n^k$ edges, where $1/n\ll \mu \ll p$. Recently, Reiher and Schacht~\cite{locally-clique} gave the minimum degree condition forcing a $K_k$-factor in locally dense graphs. Staden and Treglown~\cite{bandwidth} proved a version of the Bandwidth theorem for locally dense graphs. Similar results on perfect matchings and tilings in $k$-graphs have been considered by Han~\cite{PMtilings}.

Other notions of quasi-randomness in $k$-graphs have also been studied extensively.
One prime example is the work of Keevash~\cite{keevash-design-1} and Glock et al.~\cite{glock-design-2016} on the Existence Conjecture of Block Designs.
In particular, both of their results apply to $k$-graphs under a significantly stronger quasi-randomness condition than the one considered in this paper (a constant number of $(k-1)$-sets have good joint degrees as one expects in random $k$-graphs).

At last, although (following Lenz--Mubayi~\cite{Lenz2016Perfect}) we defined the quasi-randomness with only one-sided inequalities, all known constructions are based on random constructions, which provide two-sided edge counts.
That is, current results in this line can not be improved by requiring two-sided inequalities in the definition of quasi-randomness e.g. in~\eqref{eq:count} or~\eqref{eq:count1}. 
On the other hand, for every notion of quasi-randomness that has been studied, Lenz and Mubayi~\cite[Theorem 1]{lenz2016perfect-II} identified a large family of $k$-graphs $F$ so that every quasirandom $k$-graph $H$ admits an $F$-factor.

\subsection*{Organization}
The rest of this paper is organized as follows.
In the next section, we introduce proof ideas and give a probabilistic construction to prove Theorem~\ref{construction}. 
We give a proof of Theorem~\ref{dot-edge-partite} in Section~\ref{dot-edge-partite-thm}.
In Section~\ref{tools}, we prove an absorbing lemma, which is the main tool in the proofs of Theorem~\ref{1/8-dot-codegree}, Theorem~\ref{factor3-2} and Theorem~\ref{gcd}. We review the hypergraph regularity method in Section~\ref{regular-lemma}.
The following three sections contain proofs of Theorem~\ref{1/8-dot-codegree}, Theorem~\ref{factor3-2} and Theorem~\ref{gcd} respectively. At the end of Section~\ref{3-partite3-pf}, we will give the proof of Theorem~\ref{3-partite3-graph} using  a probabilistic construction combined with  Theorem~\ref{1/8-dot-codegree} and Theorem~\ref{factor3-2}.   
We include some discussions and further problems in the last section.

\section{Proof ideas and a proof of Theorem~\ref{construction}}

\subsection{Proof ideas}
Although our constructions of $F$-factors follow the popular framework of the absorption method (indeed a variant developed by the second author~\cite{Han2017Perfect}), 
the proof does contain novel components and reveals interesting phenomena in quasi-random $k$-graphs.
As mentioned above, for $\mathscr{S}\in \{\points, \edge\}$, $1\leq s \leq k-1$ and $k$-graphs $F$ with $F\in \factor_k^{\mathscr{S} ,s}$,
 one variant of the Tur\'an problem  needs to be considered:
\begin{enumerate}
	\item[{\rm ($\square$)}] $F\in \cover_{k}^{\mathscr{S} ,s}$, that is, every vertex must be contained in a copy of $F$ in $H$.
\end{enumerate}
When $s=1$ which is the minimum degree condition, we know that the property  {\rm ($\square$)} is enough to guarantee the existence of $F$-factors by the Theorems~\ref{factor-cover1} and ~\ref{factor-cover2}. However, for $s>1$, we need $F\in \cover_{k,>}^{\mathscr{S} ,s}$ to guarantee the existence of $F$-factors in our proofs.
In addition,  the verification of {\rm ($\square$)} is straightforward in our proofs for $k$-partite $k$-graphs (for general $k$-graphs this is very challenging and actually the bottleneck to a solution of Problem~\ref{prob} for $k\ge 4$).
For Theorem~\ref{factor3-2}, another variant of the Tur\'an problem also needs to be considered:
\begin{enumerate}
\item[{\rm ($\diamondsuit$)}] $F\in \trans^2_3$, from which we will deduce that given any bipartition $(X, Y)$ of $V(H)$, there are certain types of embeddings of $F$ that give rise to the existence of a \emph{transferral}.
\end{enumerate}
The   point will be made precise in Section~\ref{3-partite3-pf}.
Both these two properties are clearly strengthenings of the Tur\'an problem and are crucial in the construction of absorbing sets.
For the study of {\rm ($\diamondsuit$)} in Theorems~\ref{1/8-dot-codegree} and~\ref{factor3-2},
we use the regularity method and embedding techniques pioneered by Reiher--R\"odl--Schacht~\cite{Reiher2018Hypergraphs} in quasi-random $3$-graphs (see Sections~\ref{1/8} and~\ref{3-partite3-pf}).

\subsection{A proof of Theorem~\ref{construction}}
Now we prove Theorem~\ref{construction} using the following example. \\
\noindent {\bf Construction.} For $n\in \mathbb{N}$, define a probability distribution $H(n)$ on $3$-graphs of order $n$ as follows.
Let $G\in \mathbb{G}(n-1,q)$ be the random graph on $n-1$ vertices with $q\in(0,1)$.
Now let the vertex set of $H(n)$ be $V(G)\cup \{z\}$ where $z$ is a new vertex and we define $E(H(n))$ as follows. For each $3$-set $e\subseteq V(G)$, if $G[e]$ is a triangle in $G$, then add $e$ into $E(H(n))$; for each pair $\{u,v\}\subseteq V(G)$, if $\{u,v\}\notin E(G)$, then
add $\{u,v,z\}$ into $E(H(n))$.

Therefore, for each $3$-set $e$ not containing $z$, we have $e \in E(H(n))$ with probability $q^3$. Let $X_1, X_2, X_3\subset V(H(n)) \setminus \{z\}$, then the expected value of $e_{\points}(X_1,X_2, X_3)$ is at least $q^3|X_1|(|X_2|-1)(|X_3|-2)$.
For each vertex $w\in V(H(n))\setminus \{z\}$, the expected value of $\deg(w)$ in $H(n)$ is at least $q^3\binom{n-2}{2}$, and the expected value of $\deg(z)$ in $H(n)$ is $(1-q)\binom{n-1}{2}$.
Let $\alpha=\min\{\frac{q^3}{4},\frac{1-q}{4}\}$. By concentration inequalities (e.g. Janson's inequality) and the union bound, for every $\mu>0$ and sufficiently large $n$, there exists $H\in H(n)$ such that $H$ is an $(n,q^3,\mu,\points)$ $3$-graph and $\delta_1(H)\ge \alpha n^2$.

Now let us consider the vertex $z$ in $H$. Suppose that there exists a copy $F'$ of $F$ in $H$ and an isomorphic mapping  $g: V(F)\to V(F')$ such that $g(v)=z$. Recalling the condition of $F$, there is a vertex $u\in V(F)$ such that $N_F(u)\cap N_F(v)\neq \emptyset$. Then $N_H(g(u))\cap N_H(z)\neq \emptyset$, which
contradicts with the construction of $H$ that $N(z) \cap N(w)=\emptyset$ for any $w\in V(H)$.
Therefore, $H$ does not have an $F$-factor.

\section{The proof of Theorem~\ref{dot-edge-partite}}\label{dot-edge-partite-thm}
In this section, our aim is to prove Theorem~\ref{dot-edge-partite}.
Throughout the proof, we use a notion introduced by Lo and Markstr\"{o}m~\cite{Lo2015}. Let $H$ be a $k$-graph on $n$ vertices.
Given two vertices $u, v\in V(H)$ and a constant $\eta>0$, a $(k-1)$-set $S\in N(u)\cap N(v)$ is said to be \textit{$\eta$-good} for $\{u,v\}$ if $\deg(S)\ge \eta n$. Otherwise, $S$ is \textit{$\eta$-bad} for $\{u,v\}$. In particular, we say that $\{u, v\}$ is \textit{$\eta$-good} if the number of $\eta$-good $(k-1)$-sets for $\{u, v\}$ is at least $\eta n^{k-1}$. Otherwise, we say $\{u,v\}$ is \textit{$\eta$-bad}.
We have the following proposition.

\begin{prop}\label{x-good}
Given $0<\alpha<1$, there exists an integer $n_0$ such that the following holds for $n\ge n_0$.
Let $H$ be a $k$-graph of order $n$ with $\delta_1(H)\ge \alpha \binom{n-1}{k-1}$. If there exists $\alpha'>0$ such that all but at most $\frac{\alpha}{2}\binom{n-1}{k-1}$ $(k-1)$-sets $S\subset V(H)$ satisfy $\deg_H(S)\ge \alpha'n$, then for any $0<\eta<\frac{\alpha\alpha'}{4k!}$ and every $v\in V(H)$, there exists $u\in V(H)\setminus \{v\}$ such that $\{u,v\}$ is $\eta$-good in $H$.
\end{prop}

\begin{proof}
Suppose that $H$ is a $k$-graph satisfying the above conditions.
Let
\[
\mathcal{S}=\left\{S\in [V]^{k-1} : \deg(S)<\alpha'n \right\}.
\]
Then $|\mathcal{S}|<\frac{\alpha}{2}\binom{n-1}{k-1}$. For every $v\in V(H)$, we have $|N(v)|\ge \alpha \binom{n-1}{k-1}$ since $\delta_1(H)\ge \alpha \binom{n-1}{k-1}$.
Let $N^*(v):=N(v)\setminus \mathcal{S}$. Then $|N^*(v)|\ge \frac{\alpha}{2}\binom{n-1}{k-1}$ and every $(k-1)$-set $S\in N^*(v)$ satisfies $\deg(S)\ge \alpha'n$. Assume that for any $u\in V(H)\setminus \{v\}$, $\{u,v\}$ is \textit{$\eta$-bad}, namely $|N(u)\cap N^*(v)|<\eta n^{k-1}$. By double counting, we have
\[
\frac{\alpha}{2}\binom{n-1}{k-1} \alpha'n \le \sum_{S\in N^{*}(v)}{\deg(S)}< | V(H)\setminus \{v\}|\eta n^{k-1}+|N^{*}(v)|,
\]
which yields $\eta>\frac{\alpha\alpha'}{4k!}$, a contradiction. 
Therefore, there exists a vertex $u\in V(H)\setminus \{v\}$ such that $\{u,v\}$ is $\eta$-good in $H$ for all $0<\eta<\frac{\alpha\alpha'}{4k!}$.
\end{proof}

Let $K^k_k(m)$ be a complete $k$-partite $k$-graph with $m$ vertices in each part.
Our proof of Theorem~\ref{dot-edge-partite} is based on the following special form of a lemma from~\cite{Lo2015}, which tells that if $\{u,v\}$ is $\eta$-good, then $u$ is in a copy of $K^k_k(m)$.

\begin{lemma}{\rm{(}\cite[Lemma 4.2]{Lo2015}\rm{)}}.\label{good-reach.}
Let $k,m\ge 2$ be integers and $\eta\in (0,1)$. There exists a constant $\beta_0=\beta_0(k,m,\eta)>0$ and an integer $n_0=n_0(k,m,\eta)$ such
that for every $k$-graph $H$ of order $n\ge n_0$, if $\{u,v\}$ is $\eta$-good in $H$ for $u,v\in V(H)$, then there are at least $\beta_0 n^{km-1}$ $(km-1)$-sets $W$ such that both $H[\{u\}\cup W]$ and $H[\{v\}\cup W]$ contain $K^k_k(m)$. In particular, $u$ is  in at least $\beta_0 n^{km-1}$ copies of $K^k_k(m)$.
\end{lemma}

Theorem~\ref{dot-edge-partite} easily follows from Theorem~\ref{factor-cover2}, Proposition~\ref{x-good} and Lemma~\ref{good-reach.}.

\begin{proof}[Proof of Theorem~\ref{dot-edge-partite}]
Suppose that $F$ is a $k$-partite $k$-graph with vertex classes $X_1,\dots,X_k$.
By Theorem~\ref{factor-cover2}, it suffices to show that $F\in \cover^{\edge,1}_{k}$.
Let $m:=\max\{|X_i|: i\in[k]\}$, then $F$ is a sub-hypergraph of $K^k_k(m)$. Next, our goal is to show that given $0<p,\alpha<1$, we take $\mu>0$ small enough and $n$ large enough such that every vertex $v$ in an $(n,p,\mu,\edge)$ $k$-graph $H$ with $\delta_1(H)\geq \alpha n^{k-1}$ is contained in a copy of $K^k_k(m)$.

Given $0<p,\alpha<1$, we select $0<\mu<\eta<\frac{\alpha p}{8k!}$. Next we apply Lemma~\ref{good-reach.} and get the returned constant $\beta_0=\beta_0(k,m,\eta)>0$. Let $H$ be an $(n,p,\mu,\edge)$ $k$-graph with $\delta_1(H)\geq \alpha n^{k-1}$ and $v(F)\mid n$ and set $V:=V(H)$. Let
\[
\mathcal{S}=\left\{S\in [V]^{k-1} : \deg(S)<\frac{p}{2}n \right\}.
\]
Since $H$ is $(p,\mu, \edge)$-dense, we obtain that
\[
p|\mathcal{S}||V|-\mu n^k\le e_{\edge}(\mathcal{S},V)<|\mathcal{S}|\frac{p}{2}n,
\]
which implies that $|\mathcal{S}|<\frac{2\mu}{p}n^{k-1}<\frac{\alpha}{2}\binom{n-1}{k-1}$.
We apply Proposition~\ref{x-good} with $\alpha'=p/2$ to $H$ and obtain that for every $v\in V(H)$, there exists $u\in V(H)\setminus \{v\}$ such that $\{u,v\}$ is $\eta$-good in $H$. By Lemma~\ref{good-reach.}, 
$v$ is contained in a copy of $K^k_k(m)$ and we are done.
\end{proof}

\section{Tools}\label{tools}
In this section, we introduce some tools to prove Theorems~\ref{1/8-dot-codegree}, ~\ref{factor3-2} and~\ref{gcd}.
The key tool is the lattice-based absorption method developed by Han~\cite{Han2017Perfect}, which builds on the absorbing method initiated by R\"{o}dl, Ruci\'{n}ski and Szemer\'{e}di~\cite{R2015A}.
Throughout the paper, we write  $a\ll b\ll c$ to mean that we can choose the positive constants  $a, b, c$ from right to left. More precisely, there are increasing functions
$f$ and $g$ such that, given $c$, whenever we choose some $b\le  f(c)$ and $a\le g(b)$, the subsequent
statement holds. Hierarchies of other lengths are defined similarly.

The absorption approach splits the proofs into the following two parts: one is on finding an almost perfect $F$-tiling in $H$ by an almost perfect tiling lemma (Lemma~\ref{Almost-factor}), and the other is on ``finishing up" the perfect $F$-tiling by an absorbing lemma (Lemma~\ref{Absorbing-Lemma}).

\begin{lemma}{\rm(Almost Perfect Tiling {\cite[Lemma 2.1]{DHSWZ}}).} \label{Almost-factor}
Given $0< p,\alpha <1$ and a $k$-graph $F$ satisfying $\pi_{\points}(F)=0$, for any $0 <\omega < 1$, there exists an $n_0$ and $\mu > 0$ such that the following holds for $n\geq n_0$. If $H$ is an $(n,p,\mu,\points)$ $k$-graph, then there exists an $F$-tiling that covers all but at most $\omega n$ vertices of $H$.
\end{lemma}

\subsection{Absorbing lemma}
Before giving the absorbing lemma, we need some definitions introduced by Keevash and Mycroft~\cite{Ke2015}.

Let $r, f, k$ be integers with $f\geq k\geq 2$ and let $F$ be a $k$-graph of order $f$. Suppose that $H$ is a $k$-graph with a partition $\mathcal{P} =\{V_0,V_1,\dots,V_r\}$ of $V (H)$.
The \textit{index vector} $\mathbf{i}_{\mathcal{P}}(S)\in \mathbb{Z}^r$ of a subset $S\subseteq V(H)$ w.r.t. (with respect to)  $\mathcal{P}$ is the vector whose coordinates are the sizes of the intersections of $S$ with $V_1,\dots,V_r$, that is, $\mathbf{i}_{\mathcal{P}}(S)=(|S\cap V_1|, \dots, |S\cap V_r|)$ and $V_0$ is not a part of $\mathbf{i}_{\mathcal{P}}(S)$.
We call an index vector $\mathbf{i}_{\mathcal{P}}(S)$ an \emph{$s$-vector} if all its coordinates are non-negative and their sum is $s$.
Given $\lambda>0$, an $f$-vector $\mathbf{v} \in \mathbb{Z}^r$ is called a \emph{$\lambda$-robust $F$-vector} if at least $\lambda (v(H))^f$ copies $F'$ of $F$ in $H$ satisfy $\mathbf{i}_{\mathcal{P}}(V(F'))=\mathbf{v}$. Let $I^{\lambda}_{\mathcal{P},F}(H)$ be the set of all $\lambda$-robust $F$-vectors and $L^{\lambda}_{\mathcal{P},F}(H)$ be the lattice generated by the vectors of $I^{\lambda}_{\mathcal{P},F}(H)$. For $j \in [r]$, let $\mathbf{u}_j\in \mathbb{Z}^r$ be the $j^{th}$ unit vector, namely, $\mathbf{u}_j$ has $1$ on the $j^{th}$ coordinate and $0$ on other coordinates.
To ease notation,  throughout the paper, for a $k$-graph $H$ if we consider a partition $\mathcal{P}=\{V_0, V_1, V_2\}$ of $V(H)$ satisfying  $V_0=\emptyset$, then we write $X=V_1$ and $Y=V_2$, i.e.,  $\{X,Y\}:=\{ V_1, V_2\}$.

\begin{lemma}[Absorbing Lemma]\label{Absorbing-Lemma}
Suppose that $1/n\ll \mu, \gamma'\ll \lambda \ll \zeta \ll \gamma \ll  p,\alpha <1$ and $f,k, n \in \mathbb{N}$ with $f\ge k\ge 3$.
Let $F\in \cover_{k,>}^{\points, k-1}$  with $f:=v(F)$ and $H$ be an $(n, p,\mu,\points)$ $k$-graph with $\delta_{k-1}(H)\ge \alpha n$. 
If each partition $\mathcal{P}=\{V_0, V_1, V_2\}$ of $V(H)$ with $|V_i| \ge \zeta n$ for $i=1, 2$ satisfies $(1,-1) \in L^{\lambda}_{\mathcal{P},F}(H)$, 
then there exists a vertex set $W \subseteq V (H)$ with $|W| \leq \gamma n$ such that for any vertex set $U\subseteq V(H)\setminus W$ with $|U|\leq \gamma' n$ and $|U|\in f\mathbb{N}$, both $H[W]$ and $H[W \cup U]$ contain $F$-factors.
\end{lemma}

To prove Lemma~\ref{Absorbing-Lemma}, we use the notion of reachability introduced by Lo and Markstr\"{o}m~\cite{Lo2015}. Let $H$ be a $k$-graph on $n$ vertices.
Given a $k$-graph $F$ of order $f$, a constant $\beta> 0$ and an integer $i \ge 1$, we say that two vertices $u, v$ in $H$ are \textit{$(F, \beta, i)$-reachable} (in $H$) if there are at least $\beta n^{if-1}$ $(if-1)$-sets $W$ such that both $H[\{u\}\cup W]$ and $H[\{v\}\cup W]$ contain $F$-factors. Moreover, we call $W$ a \emph{reachable set} for $\{u,v\}$.
Given a vertex set $U\subseteq V(H)$, $U$ is said to be \emph{$(F, \beta, i)$-closed} if every two vertices in $U$ are $(F, \beta, i)$-reachable in $H$.
For $v\in V(H)$, let $\tilde{N}_{F,\beta, i}(v)$ be the set of vertices that are $(F,\beta, i)$-reachable to $v$.

Our proof of Lemma~\ref{Absorbing-Lemma} is based on the following lemma.

\begin{lemma}{\rm({\cite[Lemma 2.9]{DHSWZ}}).}~\label{absorb}
Suppose that $1/n \ll \gamma' \ll \gamma , \beta'_0 <1$ and $i'_0, k, f, n \in \mathbb{N}$. Let $F$ be a $k$-graph with $f:=v(F)$.
Suppose $H$ is a $k$-graph on $n$ vertices with the following two properties.
\begin{enumerate}[label=$(\roman*)$]
  \item[{\rm ($\clubsuit$)}] For any $v\in V(H)$, there are at least $\gamma n^{f-1}$ copies of $F$ containing it; \label{item:1}
  \item[{\rm ($\spadesuit$)}]  there exists $V'_0\subseteq V(H)$ with $|V'_0|\leq \gamma^2n$ such that $V (H)\setminus V'_0$ is $(F, \beta'_0, i'_0)$-closed in $H$.\label{item:2}
\end{enumerate}
Then there exists a vertex set $W$ with $V'_0\subseteq W\subseteq V(H)$ and $|W|\leq  \gamma n$ such that for any vertex set $U \subseteq V (H)\setminus W$ with $|U| \leq \gamma' n$ and $|U| \in f\mathbb{N}$, both $H[W]$ and $H[U \cup W]$ contain $F$-factors. 
\end{lemma}

It is worth mentioning that Han--Zang--Zhao~\cite[Lemma 3.6]{Han2017Minimum} proved Lemma~\ref{absorb} for $F$ being a complete 3-partite 3-graph.
Note that to prove Lemma~\ref{Absorbing-Lemma}, it suffices to show that ($\clubsuit$) and ($\spadesuit$) hold for $k$-graphs $H$ satisfying the conditions of Lemma~\ref{Absorbing-Lemma}. We first consider the property ($\clubsuit$). 
Recall that the family $\cover_{k,>}^{\points,s}$ consists of $k$-graphs $F$ satisfying that for all $0< p, \alpha<1$, there is some $n_0\in \mathbb{N}$ and $\gamma, \mu > 0$ so that if $H$ is an $(n, p, \mu, \points)$ $k$-graph $H$ with $\delta_{s}(H)\geq \alpha n^{k-s}$ and $n \geq n_0$, then each vertex $v$ in $V(H)$ is contained in at least $\gamma n^{v(F)-1}$ copies of $F$. 
Therefore, the property ($\clubsuit$) follows from the condition $F\in \cover_{k,>}^{\points,k-1}$ in Lemma~\ref{Absorbing-Lemma}. 

Next we consider the property ($\spadesuit$). Following the approach in~\cite{Han2017Minimum}, given a $k$-graph $H$, we first find a partition of $V(H)$ such that all parts are $(F, \beta, i)$-closed in $H$ and then study the reachability between different parts. The following lemma provides such a partition.

\begin{lemma}{\rm (\cite[Theorem 6.3]{Han_2020})}.
\label{partation}
Suppose that $1/n\ll \beta_0\ll \beta \ll \delta, 1/c, 1/f <1$ and $c, f,k, n \in \mathbb{N}$ with $f\ge k\ge 3$.
 Let $F$ be a $k$-graph with $f:=v(F)$. Assume $H$ is a $k$-graph on $n$ vertices and $S\subseteq V(H)$ is such that $|\tilde{N}_{F, \beta , 1}(v)\cap S|\geq \delta n$ for any $v\in S$. Further, suppose every set of $c+1$ vertices in $S$ contains two vertices that are $(F, \beta , 1)$-reachable in $H$. Then we can find a partition $\mathcal{P}$ of $S$ into $V_1,\dots,V_r$ with $r \leq \min\{c,\lfloor1/{\delta}\rfloor\}$ such that for any $i \in [r]$, $|V_i|\geq (\delta -\beta )n$ and $V_i$ is $(F, \beta_0 , 2^{c-1})$-closed in $H$. 
\end{lemma}

In order to use Lemma~\ref{partation} later, we need the following two lemmas.

\begin{lemma}{\rm (\cite[Lemma 2.6]{DHSWZ})}. \label{S-closed}
Given integers $c, k, f\geq 2$ and a constant $0 < \beta  <1$, there exist $1/n \ll \delta \ll 1/c$ such that the following holds.
Let $F$ be a $k$-graph on $f$ vertices. Assume that $H$ is a $k$-graph on $n$ vertices satisfying that every set of $c+1$ vertices contains two vertices that are $(F,\beta ,1)$-reachable in $H$. Then there exists $S\subseteq V(H)$ with $|S|\geq (1-c\delta )n$ such that $|\tilde{N}_{F, \beta , 1}(v)\cap S|\geq \delta n$ for any $v\in S$.
\end{lemma}

\begin{lemma}\label{c+1-reachable}
Given integers $k, f\geq 2$ and a constant $0 <\gamma <1$, there exist $1/n \ll \beta  \ll \gamma$ such that the following holds. Let $F$ be a $k$-graph on $f$ vertices.
 If $H$ is a $k$-graph on $n$ vertices satisfying that every vertex $v$ in $V(H)$ is contained in at least $\gamma n^{f-1}$ copies of $F$, then every set of $\lfloor 1/\gamma \rfloor+1$ vertices in $V(H)$ contains two vertices that are $(F,\beta ,1)$-reachable in $H$.
\end{lemma}

\begin{proof}
Set $c:=\lfloor 1/\gamma \rfloor$ and choose $\beta $ such that $(c+1)\gamma>1+(c+1)^2\beta $. For a vertex $v\in V(H)$, let $C_F(v)$ be the family of copies of $F$ containing $v$.
Given any $c+1$ vertices $v_1, \dots , v_{c+1}$, since each of them is contained in at least 
$\gamma n^{f-1}$
copies of $F$, we have $\sum_{i=1} ^{c+1}|C_F(v_i)|\geq (c+1)\gamma n^{f-1}>(1+(c+1)^2\beta )n^{f-1}$. This implies that there exist two vertices $u,w$ such that there are at least $\beta n^{f-1}$ $(f-1)$-sets $W$ such that both $H[\{u\}\cup W]$ and $H[\{w\}\cup W]$ are copies of $F$ in $H$. Namely, they are $(F,\beta,1)$-reachable in $H$.
\end{proof}

By Lemma~~\ref{partation} and the conditions of $H$ in Lemma~\ref{Absorbing-Lemma}, the property ($\spadesuit$) can be obtained by the following lemma directly.

\begin{lemma}{\rm{(\cite[\rm Lemma 3.9]{Han2017Minimum})}.}
\label{closed}
Let $i_0, k, r,f> 0$ be integers and let $F$ be a $k$-graph with $f:=v(F)$. Given constants $\zeta, \beta_0 , \lambda > 0$, there exists $\beta_0'>0$ and integers $i'_0, n_0$ such that the following holds for all $n\geq n_0$. Let $H$ be a $k$-graph on $n$ vertices with a partition $\mathcal{P} =\{V_0,V_1,\dots,V_r\}$ of $V (H)$ such that for each $j\in [r]$, $|V_j|\geq \zeta n$ and $V_j$ is $(F,\beta_0 , i_0)$-closed in $H$. If $\mathbf{u}_j -\mathbf{u}_l \in  L^{\lambda}_{\mathcal{P},F}(H)$ for all $1 \leq  j <l \leq r$, then $V(H)\setminus V_0$ is $(F, \beta'_0, i'_0)$-closed in $H$. 
\end{lemma}

\begin{proof}[Proof of Lemma~\ref{Absorbing-Lemma}]
Let  $f\ge k\ge 3$ be integers. We choose constants satisfying the following hierarchy:
\[
1/n \ll \mu, \gamma'  \ll \beta'_0 \ll \lambda, \beta_0 \ll \zeta, \beta \ll \delta\ll   \gamma \ll  p,\alpha, 1/f <1,
\]
$\delta< \gamma^3$ and $\delta-\beta>\zeta $.
Let $F\in \cover_{k,>}^{\points, k-1}$ with $f:=v(F)$ and $H$ be an $(n, p,\mu,\points)$ $k$-graph with $\delta_{k-1}(H)\ge \alpha n$ satisfying conditions in Lemma~\ref{Absorbing-Lemma}. By Lemma~\ref{absorb}, it suffices to verify that $H$ satisfies properties ($\clubsuit$) and ($\spadesuit$). As mentioned above, $H$ satisfies property ($\clubsuit$) since $F\in \cover_{k,>}^{\points, k-1}$. 
Since every vertex $v$ in $V(H)$ is contained in at least $\gamma n^{f-1}$ copies of $F$, every set of $\lfloor 1/\gamma \rfloor+1$ vertices in $V(H)$ contains two vertices that are $(F,\beta ,1)$-reachable in $H$ by Lemma~\ref{c+1-reachable}.
Set $c:=\lfloor 1/\gamma \rfloor$. Next by Lemma~\ref{S-closed}, there exists $S\subseteq V(H)$ such that $|S|\geq (1-c\delta )n\ge  (1-\gamma^2)n$ and $|\tilde{N}_{F, \beta , 1}(v)\cap S|\geq \delta n$ for any $v\in S$.
Following Lemma~\ref{partation}, we get a partition $\mathcal{P}'$ of $S$ into $V_1,\dots,V_r$ with $r \leq \min\{c,\lfloor1/{\delta}\rfloor\}$ such that for any $i \in [r]$, $|V_i|\geq (\delta -\beta )n>\zeta n $ and $V_i$ is $(F, \beta_0 , 2^{c-1})$-closed in $H$.
Let $\mathcal{P}''=\{V'_0, V_1,\dots,V_r \}$ with $V'_0=V(H)\setminus S$. Clearly, $|V'_0|\le \gamma^2 n$.
Recalling the conditions in Lemma~\ref{Absorbing-Lemma}, for all $1 \leq  j <l \leq r$, $\mathcal{P}=\{V_0, V_j, V_l\}$ with $V_0=\bigcup_{i\in [r]\setminus \{j,l\}}V_i$ satisfies $(1,-1)\in L^{\lambda}_{\mathcal P,F}(H)$. We have $\mathbf{u}_j -\mathbf{u}_l\in L^{\lambda}_{\mathcal P'',F}(H)$. Finally, we apply Lemma~\ref{closed} with $i_0=2^{c-1}$ and partition $\mathcal{P}''$. Therefore, the property ($\spadesuit$) holds for $H$ by  Lemma~\ref{closed}.
\end{proof}

At last, by combining Lemma~\ref{Almost-factor} and Lemma~\ref{Absorbing-Lemma} we obtain the following corollary.

\begin{cor}\label{covertrans}
Suppose that $1/n\ll \mu\ll \lambda' \ll \zeta\ll   p,\alpha, 1/f <1$ and $f,k, n \in \mathbb{N}$ with $f\ge k\ge 3$ and $n\in f\mathbb{N}$.
Let $F\in \cover_{k,>}^{\points, k-1}$ 
with $f:=v(F)$  
and $\pi_{\points}(F)=0$.
Suppose $H$ is an $(n, p,\mu,\points)$ $k$-graph with $\delta_{k-1}(H)\ge \alpha n$ and the following property.
\begin{enumerate}[label=$(\roman*)$]
  \item[{\rm ($\heartsuit$)}]  For each induced subhypergraph $H'$ on $V'\subseteq V(H)$ with $|V'|\ge 2\zeta n$, each partition $\mathcal{P'}=\{X, Y\}$ of $V'$ with $|X|, |Y|\ge \zeta n$ satisfies $(1,-1) \in L^{\lambda'}_{\mathcal{P'},F}(H')$.\label{item:3}
\end{enumerate}
Then $H$ has an $F$-factor. 
\end{cor}

\begin{proof}
 We choose constants satisfying the following hierarchy:
\[
 1/n\ll \mu, \gamma' \ll \lambda' \ll \zeta \ll \gamma \ll \alpha, p, 1/f < 1.
\]
Let $n\in f\mathbb{N}$ and $F\in \cover_{k,>}^{\points, k-1}$ with $f:=v(F)$ and $\pi_{\points}(F)=0$. Suppose that $H$ is an $(n,p,\mu,\points)$ $k$-graph with $\delta_{k-1}(H)\ge \alpha n$ and $\mathcal{P}=\{V_0,V_1,V_2\}$ is a partition of $V(H)$ with $|V_i|\geq \zeta n$ for $i=1,2$. 
Set $H':=H[V_1\cup V_2]$. By  the property ($\heartsuit$), the partition $\mathcal P'=\{V_1, V_2\}$ of $V(H')$ satisfies $(1,-1) \in L^{\lambda'}_{\mathcal{P'},F}(H')$ which implies that $(1,-1)\in L^{\lambda}_{\mathcal{P},F}(H)$, where $\lambda:=\lambda'(2\zeta)^f$.
By Lemma~\ref{Absorbing-Lemma}, we can find an absorbing set $W\subseteq V(H)$ with $|W|\leq \gamma n$.
Let $H'':=H[V(H)\setminus W]$ and  $\mu_1=\mu/(1-\gamma)^k$. Note that $H''$ is an $(n-|W|,p,\mu_1, \points)$ $k$-graph since $|W|\leq \gamma n$ and
\[ 
\mu n^k \le \frac{\mu}{(1-\gamma)^k}v(H'')^k= \mu_1 v(H'')^k.
\]
Since $\pi_{\points}(F)=0$, applying Lemma~\ref{Almost-factor} on $H''$ with $\omega=\gamma'$, we obtain an $F$-tiling that covers all but a set $U$ of at most $\gamma' n$ vertices. By the absorbing property of $W$, $H[W\cup U]$ contains an $F$-factor and thus we obtain an $F$-factor of $H$.
\end{proof}

\section{Hypergraph Regularity Method}\label{regular-lemma}
In order to prove Theorem~\ref{1/8-dot-codegree} and Theorem~\ref{factor3-2}, we will use the hypergraph regularity method for $3$-graphs to verify the conditions of Corollary~\ref{covertrans}.
Here we follow the approach from R\"{o}dl and Schacht~\cite{MR2351689,regular-lemmas2007}, combined with results from~\cite{CFKO} and~\cite{K2010Hamilton}.

The central concepts of the hypergraph regularity lemma are \emph{regular complex} and \emph{equitable partition}. Before we state the hypergraph regularity lemma, we introduce some necessary notation.
For reals $x,y,z$ we write $x = y \pm z$ to denote that $y-z \leq x \leq y+z$.

\subsection{Regular complexes}
A \emph{hypergraph} $\mathcal{H}$ consists of a vertex set $V(\mathcal{H})$ and an edge set $E(\mathcal{H})$, where every edge $e \in E(\mathcal{H})$ is a non-empty subset of $V(\mathcal{H})$. So a $3$-graph as defined earlier is a $3$-uniform hypergraph in which every edge has size $3$. A hypergraph $\mathcal{H}$ is a \emph{complex} if whenever $e \in E(\mathcal{H})$ and $e'$ is a non-empty subset of $e$ we have that $e'\in E(\mathcal{H})$. All the complexes considered in this paper have the property that all vertices are contained in an edge. A complex $\mathcal{H}^{\leq 3}$ is a \emph{$3$-complex} if all the edges of $\mathcal{H}^{\leq 3}$ consist of at most $3$ vertices. Given a $3$-complex $\mathcal{H}^{\leq 3}$, for each $i \in \{1,2,3\}$, the edges of size $i$ are called \emph{$i$-edges} of $\mathcal{H}^{\leq 3}$ and we denote by $H^{(i)}$ the \emph{underlying $i$-graph} of $\mathcal{H}^{\leq 3}$: the vertices of $H^{(i)}$ are those of $\mathcal{H}^{\leq 3}$ and the edges of $H^{(i)}$ are the $i$-edges of $\mathcal{H}^{\leq 3}$. Note that a $3$-graph $H$ can be turned into a $3$-complex by making every edge into a \emph{complete $i$-graph} $K^{(i)}_3$ (i.e., consisting of all $\binom{3}{i}$ different $i$-tuples on $3$ vertices), for each $ i\in \{1,2,3\}$.

Given positive integers $s\geq 3$ and $i\in \{1,2,3\}$, an \emph{$(s,i)$-graph} $H^{(i)}_s$ is an $s$-partite $i$-graph, by which we mean that the vertex set of $H^{(i)}_s$ can be partitioned into sets $V_1,\dots, V_s$ such that every edge of $H^{(i)}_s$ meets each $V_j$ in at most one vertex for $j\in [s]$. For every $\Lambda_i\in [s]^i$, let $H^{(i)}_{s}[\cup_{\lambda\in \Lambda_i }V_{\lambda}]$ denote the  subgraph of $H^{(i)}_{s}$ induced by  $\cup_{\lambda\in \Lambda_i }V_{\lambda}$.
Similarly, an \emph{$(s,i)$-complex} $\mathcal{H}^{\leq i}_s$ is an $s$-partite $i$-complex.

Given $i \in\{2,3\}$, let $H^{(i)}_i$ and $H^{(i-1)}_i$ be on the same vertex set. We denote by $\mathcal{K}_i(H^{(i-1)}_i)$ the family of $i$-sets of vertices which form a copy of the complete $(i-1)$-graph $K^{(i-1)}_i$ in $H^{(i-1)}_i$. We define the \emph{density} of $H^{(i)}_i$ w.r.t. $H^{(i-1)}_i$ to be
\[
d(H^{(i)}_i|H^{(i-1)}_i):= \begin{cases}
\frac{|E(H^{(i)}_i)\cap \mathcal{K}_i(H^{(i-1)}_i)|}{|\mathcal{K}_i(H^{(i-1)}_i)|} &\text{if\ } |\mathcal{K}_i(H^{(i-1)}_i)|>0,\\
0&\text{otherwise}.
\end{cases}
\]
More generally, if $\mathbf{Q}:= (Q(1), Q(2),\dots, Q(r))$ is a collection of $r$ subgraphs of $H^{(i-1)}_i$, we define $\mathcal{K}_i(\mathbf{Q}):= \bigcup^r_{j=1}\mathcal{K}_i(Q(j))$ and
\[
d(H^{(i)}_i|\mathbf{Q}):= \begin{cases}
\frac{|E(H^{(i)}_i)\cap \mathcal{K}_i(\mathbf{Q})|}{|\mathcal{K}_i(\mathbf{Q})|} &\text{if\ } |\mathcal{K}_i(\mathbf{Q})|>0,\\
0&\text{otherwise}.
\end{cases}
\]

We say that an $H^{(i)}_i$ is \emph{$(d_i, \delta, r)$-regular} w.r.t. an $H^{(i-1)}_i$ if every $r$-tuple $\mathbf{Q}$ with $|\mathcal{K}_i(\mathbf{Q})| \geq \delta|\mathcal{K}_i(H^{(i-1)}_i)|$ satisfies $d(H^{(i)}_i|\mathbf{Q}) = d_i \pm \delta$. Instead of $(d_i, \delta, 1)$-regular, we refer to $(d_i, \delta)$-\emph{regular}.
Moreover, for $i\in \{2,3\}$ and $s \geq i$, we say that $H^{(i)}_{s}$ is \emph{$(d_i, \delta,r)$-regular} w.r.t. $H^{(i-1)}_{s}$ if for every $\Lambda_i\in  [s]^i$ the restriction $H^{(i)}_{s}[\Lambda_i]=H^{(i)}_{s}[\cup_{\lambda\in \Lambda_i }V_{\lambda}]$ is \emph{$(d_i, \delta,r)$-regular} w.r.t. the restriction $H^{(i-1)}_{s}[\Lambda_i]=H^{(i-1)}_{s}[\cup_{\lambda\in \Lambda_i }V_{\lambda}]$.

\begin{defi}[$(d_2, d_3, \delta_3, \delta, r)$-regular complexes]
Given $s\ge 3$ and an $(s, 3)$-complex $\mathcal{H}$, we say that $\mathcal{H}$ is \emph{$(d_2, d_3, \delta_3, \delta, r)$-regular} if the following conditions hold:\vspace{3mm}

$\bullet$ for every $2$-tuple $\Lambda_2$ of vertex classes, either $H_s^{(2)}[\Lambda_2]$ is $(d_2, \delta)$-regular w.r.t $H_{s}^{(1)}[\Lambda_2]$ or 
\\
$d(H_s^{(2)}[\Lambda_2]|H^{(1)}_s[\Lambda_2]) = 0$;

$\bullet$ for every $3$-tuple $\Lambda_3$ of vertex classes either $H^{(3)}_s[\Lambda_3]$ is $(d_3, \delta_3, r)$-regular w.r.t $H^{(2)}_s[\Lambda_3]$ or $d(H^{(3)}_s[\Lambda_3]|H^{(2)}_s[\Lambda_3]) = 0$.\vspace{3mm}
\end{defi}

\subsection{Equitable partition}\label{section-equ-partition}


Suppose that $V$ is a finite set of vertices and $\mathcal{P}^{(1)}$ is a partition of $V$ into sets $V_1, \dots , V_{a_1}$, which will be called \emph{clusters}. For two disjoint sets $V_i$ and $V_j$, we denote by $K(V_i,V_j)$ the \emph{complete bipartite} graph between $V_i$ and $V_j$. We denote by  $\mathrm{Cross}_2:=\bigcup_{1\leq i<j\leq a_1}K(V_i,V_j)$ and $\mathrm{Cross}_3$ the set of all $3$-subsets of $V$ that meet each $V_i$ in at most one vertex for $1\leq i\leq a_1$. Let $\mathcal{P}^{ij}$ be a partition of $K(V_i,V_j)$. Then the partition classes of $\mathcal{P}^{ij}$ are some bipartite graphs $P^{(2)}$. Let
$\mathcal{P}^{(2)}=\cup_{1\le i<j\le a_1} \mathcal{P}^{ij}$. So $\mathcal{P}^{(2)}$ is a partition of $\mathrm{Cross}_2$ into several bipartite graphs. 


For every $\{x,y\}\in \mathrm{Cross}_2$ with $x\in V_i$ and $y\in V_j$ for $1\le i<j\le a_1$, there exists a unique bipartite graph $P^{(2)}(\{x,y\})\in \mathcal{P}^{(2)}$ so that $\{x,y\}\in P^{(2)}(\{x,y\})$. Now we give the definition of a \emph{polyad}. The polyad of $\{x,y\}$ is
\[
\hat{P}^{(1)}=\hat{P}^{(1)}(\{x,y\})=V_i\dot\cup V_j.
\]
Clearly, $\mathcal{K}_2(\hat{P}^{(1)})=K(V_i,V_j)$.  For every $3$-set $J\in \mathrm{Cross}_3$ the polyad of $J$ is:
\[
\hat{P}^{(2)}(J):=\bigcup\big\{P^{(2)}(I): I\in [J]^{2}\big\}.
\]
So we can view $\hat{P}^{(2)}(J)$ as a $3$-partite graph (whose vertex classes are clusters intersecting $J$). Let $\mathcal{\hat{P}}^{(2)}$ be the set consisting of all the $\hat{P}^{(2)}(J)$ for all $J\in \mathrm{Cross}_3$. It is easy to verify $\{\mathcal{K}_3(\hat{P}^{(2)}) : \hat{P}^{(2)}\in \mathcal{\hat{P}}^{(2)}\}$ is a partition of $\mathrm{Cross}_3$.

Given a vector of positive integers $\mathbf{a}=(a_1,a_2)$, we say that $\mathcal{P} = \{\mathcal{P}^{(1)},\mathcal{P}^{(2)}\}$ is a \emph{family of partitions} on $V$, if the following conditions hold:\vspace{3mm}

$\bullet$ $\mathcal{P}^{(1)}$ is a partition of $V$ into $a_1$ clusters.

$\bullet$ $\mathcal{P}^{(2)}$ is a partition of $\mathrm{Cross}_2$ satisfying that 
for every $P^{(2)}\in \mathcal{P}^{(2)}$, there exists $\{i,j\}\in [a_1]^2$ such that $P^{(2)} \subseteq K(V_i,V_j)$. Moreover, 
$|\{P^{(2)}\in \mathcal{P}^{(2)}: P^{(2)}\subseteq  K(V_i,V_j)\}|=a_2$ for every $\{i,j\}\in [a_1]^2$.

So for each $J \in \mathrm{Cross}_3$ we can view $\bigcup^{2}_{i=1}\hat{P}^{(i)}(J)$ as a $(3, 2)$-complex.
\vspace{3mm}

\begin{defi}[$(\eta, \delta, t)$-equitable]\label{eq-partition}
Suppose $V$ is a set of $n$ vertices, $t$ is a positive integer and $\eta, \delta>0$. We say a family of partitions $\mathcal{P}$ is \emph{$(\eta, \delta, t)$-equitable} if it satisfies the following:
\begin{enumerate}
  \item $\mathcal{P}^{(1)}$ is a partition of $V$ into $a_1$ clusters of equal size, where $1/\eta \leq a_1 \leq t$ and $a_1$ divides $n$;
  \item $\mathcal{P}^{(2)}$ is a partition of $\mathrm{Cross}_2$ into at most $t$ bipartite subgraphs;
  \item there exists $d_2$ such that $d_2\geq 1/t$ and $1/d_2 \in \mathbb{N}$;
  \item for every $3$-set $K \in \mathrm{Cross}_3$, $\hat{P}^{(2)}(K)$ is $(d_2, \delta)$-regular w.r.t. $\bigcup_{I\in [K]^{2}}\hat{P}^{(1)}(I)$.
\end{enumerate}
\end{defi}
Note that the final condition implies that the classes of $\mathcal{P}^{(2)}$ have almost equal size.

\subsection{Statement of the regularity lemma.}
Let $\delta_3 > 0$ and $r \in \mathbb{N}$. Suppose that $H$ is a $3$-graph on $V$ and $\mathcal{P}$ is a family of partitions on $V$. Given a polyad $\hat{P}^{(2)} \in \hat{\mathcal{P}}^{(2)}$, we say that $H$ is \emph{$(\delta_3, r)$-regular} w.r.t. $\hat{P}^{(2)}$ if $H$ is $(d_3, \delta_3, r)$-regular w.r.t. $\hat{P}^{(2)}$ for some $d_3$. Finally, we define that $H$ is \emph{$(\delta_3, r)$-regular} w.r.t. $\mathcal{P}$.

\begin{defi}[\emph{$(\delta_3, r)$-regular} w.r.t. $\mathcal{P}$]
We say that a $3$-graph $H$ is \emph{$(\delta_3, r)$-regular} w.r.t. $ \mathcal{P}$ if
\[
\big|\bigcup\big\{\mathcal{K}_3(\hat{P}^{(2)}) : \hat{P}^{(2)}\in \mathcal{\hat{P}}^{(2)}\\
\text{and\ } H \text{\ is\ not\ } (\delta_3, r)\text{-regular\ w.r.t.\ } \hat{P}^{(2)} \big\}\big| \le \delta_3 |V|^3.
\]
\end{defi}
This means that no more than a $\delta_3$-fraction of the $3$-subsets of $V$ form a $K_3^{(2)}$ that lies within a polyad w.r.t. which $H$ is not regular.

Now we are ready to state the regularity lemma.
\begin{thm}[Regularity lemma \rm{\cite[Theorem 17]{regular-lemmas2007}}] \label{Reg-lem}
For all positive constants $\eta$ and $\delta_3$ and all functions $r:\mathbb{N} \rightarrow \mathbb{N}$ and $\delta:\mathbb{N} \rightarrow (0,1]$, there are integers $t$ and $n_0$ such that the following holds.
For every $3$-graph $H$ of order $n\ge n_0$ and $t!$ dividing $n$, there exists a family of partitions $ \mathcal{P}$ of $V(H)$ such that\vspace{2mm}

$(1)$ $\mathcal{P}$ is $(\eta, \delta(t), t)$-equitable and

$(2)$ $H$ is $(\delta_3, r(t))$-regular w.r.t. $\mathcal{P}$.
\end{thm}

Note that the constants in Theorem~\ref{Reg-lem} can be chosen to satisfy the following hierarchy:
\[
\frac{1}{n_0}\ll \frac{1}{r}=\frac{1}{r(t)}, \delta=\delta(t)\ll \min\{\delta_3,1/t\}\ll \eta.
\]
Given $d\in (0,1)$, we say that an edge $e$ of $H$ is \emph{$d$-useful} if it lies in $\mathcal{K}_3 (\hat{P}^{(2)})$ for some $\hat{P}^{(2)} \in \hat{\mathcal{P}}^{(2)}$ such that $H$ is $(d_3, \delta_3, r)$-regular w.r.t. $\hat{P}^{(2)}$ for some $d_3\geq d$. If we choose $d\gg \eta$, then the following lemma will be helpful in later proofs.

\begin{lemma}\rm{(\cite[Lemma 4.4]{K2010Hamilton}).}\label{useful-edge}
At most $2dn^3$ edges of $3$-graph $H$ are not $d$-useful.
\end{lemma}

To ease notation, we denote by $P^{ij}_{a}=(V_i\dot\cup V_j,E^{ij}_{a})$ a specific bipartite subgraph in $K(V_i,V_j)$ and let $P^{ijk}_{abc}=P^{ij}_{a}\cup P^{ik}_{b}\cup P^{jk}_{c}$ denote a specific polyad over $2$-graphs, i.e. $ P^{ijk}_{abc}=(V_i\dot\cup V_j\dot\cup V_k,E^{ij}_{a}\cup E^{ik}_{b}\cup E^{jk}_{c})$.

\subsection{Statement of a counting lemma.}
In our proofs we shall also use a counting lemma. Before stating this lemma, we need more definitions.

Suppose that $\mathcal{H}$ is an $(s, 3)$-complex with vertex classes $V_1,\dots,V_s$, which all have size $m$. Suppose also that $\mathcal{G}$ is an $(s, 3)$-complex with vertex classes $X_1,\dots, X_s$ of size at most $m$. For $i\in\{2,3\}$, we write $E_i(\mathcal{G})$ for the set of all $i$-edges of $\mathcal{G}$ and $e_i(\mathcal{G}):= |E_i(\mathcal{G})|$. We say that $\mathcal{H}$ \emph{respects the partition} of $\mathcal{G}$ if whenever $\mathcal{G}$ contains an $i$-edge with vertices in $X_{j_1}, \dots , X_{j_i}$, then there is an $i$-edge of $\mathcal{H}$ with vertices in $V_{j_1}, \dots , V_{j_i}$. On the other hand, we say that a labelled copy of $\mathcal{G}$ in $\mathcal{H}$ is \emph{partition-respecting} if for each $i \in[s]$ the vertices corresponding to those in $X_i$ lie within $V_i$. We denote by $|\mathcal{G}|_{\mathcal{H}}$ the number of labelled, partition-respecting copies of $\mathcal{G}$ in $\mathcal{H}$.

\begin{lemma}\label{Counting-lem}{\rm(Counting lemma {\cite[Theorem 4]{CFKO}}).} Let $s,r,t,n_0$ be positive integers and
let $d_2, d_3$, $\delta,\delta_3$ be positive constants such that $1/{d_2},1/{d_3}\in \mathbb{N}$ and
\[
1/{n_0}\ll 1/r, \delta\ll \min\{\delta_3, d_2 \}\ll \delta_3 \ll d_3, 1/s, 1/t.
\]
Then the following holds for all integers $n\ge n_0$. Suppose that $\mathcal{G}$ is an $(s, 3)$-complex on $t$ vertices with vertex classes $X_1,\dots, X_s$. Suppose also that $\mathcal{H}$ is a $(d_2,d_3,\delta_3,\delta, r)$-regular $(s, 3)$-complex with vertex classes $V_1,\dots,V_s$ all of size $n$, which respects the partition of $\mathcal{G}$. Then
\[
|\mathcal{G}|_{\mathcal{H}}\ge \frac{1}{2}n^t d^{e_2(\mathcal{G})}_2d^{e_3(\mathcal{G})}_3.
\]
\end{lemma}

\subsection{Reduced Hypergraphs.}

In order to use the counting lemma, we need to know the distribution of dense and regular polyads. For this purpose we need to introduce the so-called \textit{reduced hypergraphs} for 3-graphs. The terminology below follows \cite[Section~3]{reiher2018some}.

Given a $3$-graph $H$ on $n$ vertices with $n$ sufficiently large, applying Theorem~\ref{Reg-lem} to $H$, we obtain a family of partitions $ \mathcal{P}=\{\mathcal{P}^{(1)},\mathcal{P}^{(2)}\}$ of $V(H)$ such that

$(1)$ $\mathcal{P}$ is $(\eta, \delta(t), t)$-equitable and

$(2)$ $H$ is $(\delta_3, r(t))$-regular w.r.t. $\mathcal{P}$.

By the family of partitions $ \mathcal{P}$, we define an  auxiliary $3$-graph $\mathcal{A}$ of $H$ as follows.
For each bipartite subgraphs 
$P^{(2)}$ in $\mathcal{P}^{(2)}$, we view $P^{(2)}$ as a vertex of $\mathcal{A}$, namely, the sets $\mathcal{P}^{ij}$ for $\{i,j\}\in [a_1]^2$ as the \textit{vertex classes} of $\mathcal{A}$, and the 3-partite 3-graph $\mathcal{A}^{ijk}$ as the \textit{constituents} of $\mathcal{A}$
for any three distinct indices $i, j, k \in [a_1]$. We define $\mathcal{A}^{ijk}$ with vertex classes $\mathcal{P}^{ij}$, $\mathcal{P}^{jk}$ and $\mathcal{P}^{ik}$, and a triple $\{P^{ij}_a,P^{ik}_b,P^{jk}_c\}$ is defined to be an edge of $\mathcal{A}^{ijk}$ if and only if $H$ is $(d,\delta_3,r)$-regular w.r.t. $P^{ijk}_{abc}$ for some $d\geq d_3$.
 Under such circumstances we call the $\binom{|a_1|}{2}$-partite $3$-graph $\mathcal{A}$ defined by
\[
V(\mathcal{A})=\bigcup_{\{i,j\}\in {I}^{2}}\mathcal{P}^{ij} \ \ \text{and}\ \
E(\mathcal{A})=\bigcup_{\{i,j,k\}\in {I}^{3}}E(\mathcal{A}^{ijk})
\]
a \textit{reduced hypergraph} of $H$. For any three vertex subsets $P_1,P_2,P_3\subseteq V(\mathcal{A})$, let $E_{\mathcal{A}}(P_1,P_2,P_3)$ be the set of $(p_1,p_2,p_3)\in P_1\times P_2\times P_3$ such that $\{p_1,p_2,p_3\}\in E(\mathcal{A})$. For $\eta>0$ such a reduced hypergraph $\mathcal{A}$ is said to be \emph{$\eta$-dense} if
\[
|E(\mathcal{A}^{ijk})|\geq \eta |\mathcal{P}^{ij}||\mathcal{P}^{ik}||\mathcal{P}^{jk}|
\]
holds for every triple $\{i,j,k\}\in I^{3}$.

\section{Proof of Theorem~\ref{1/8-dot-codegree}}\label{1/8}

In order to prove Theorem~\ref{1/8-dot-codegree}, it suffices to verify the conditions in Corollary~\ref{covertrans}. The subsequent two results serve this purpose.

\begin{prop}\label{cover1}
	Given  $0<\alpha<1$ and a $k$-partite $k$-graph $F$ with $f:=v(F)$,  there exists $\gamma>0$ and an $n_0\in \mathbb{N}$ such that the following holds for $n\geq n_0$.
If $H$ is an $n$-vertex $k$-graph with $\delta _{k-1}(H)\geq \alpha n$, then for any vertex $v$ in $V(H)$, $v$ is contained in more than $\gamma n^{f-1}$ copies of $F$, namely,  $F\in \cover_{k,>}^{\points ,k-1}$.
\end{prop}

\begin{proof}
Given a $k$-partite $k$-graph $F$ with vertex classes $X_1,\dots,X_k$, let $1/n \ll \mu\ll \gamma \ll \alpha<1$ and $m:=\max\{|X_i|: i\in[k]\}$.
Suppose $H$ is an $n$-vertex $k$-graph with $\delta_{k-1}(H)\ge \alpha n$.
First note that $\delta_{k-1}(H)\geq \alpha n$ implies that $\delta_1(H)\geq \alpha\binom{n-1}{k-1}$. By Proposition~\ref{x-good}, for any $0<\eta<\frac{\alpha^2}{4k!}$ and each $v\in V(H)$, there exists $u\in V(H)\setminus \{v\}$ such that $\{u,v\}$ is $\eta$-good in $H$.
By Lemma~\ref{good-reach.}, there is a constant $\gamma>0$ such that
$v$ is contained in more than $\gamma n^{f-1}$ copies of $F$, since $F$ is a subhypergraph of $K^k_k(m)$.
\end{proof}

\begin{lemma}\label{transferral}
Suppose that $1/n\ll \mu \ll \lambda \ll \varepsilon, \zeta <1$ and $f, n \in \mathbb{N}$ with $f\ge 3$.
Let $F$ be a $3$-partite $3$-graph with $f:=v(F)$. If  $H$ is an $(n,\frac{1}{8}+\varepsilon,\mu,\points)$ $3$-graph and $\mathcal{P}=\{X,Y\}$ is a partition of $V(H)$ with $|X|,|Y|\ge \zeta n$, then $(1,-1)\in L_{\mathcal{P},F}^{\lambda }(H)$.
\end{lemma}

\begin{proof}
Suppose that $F$ is a 3-partite 3-graph with vertex classes $X_1,X_2,X_3$ and $|X_i|=f_i$ for $i\in [3]$. We choose constants satisfying the following hierarchy:
\[
1/n\ll \mu \ll \lambda \ll 1/r, \delta \ll 1/t \le d_2 \ll \eta, \delta_3 \ll  d_3\ll \varepsilon ,\zeta, 1/f,
\]
where $1/d_2\in \mathbb{N}$. Recall that the hypergraph regularity lemma is proved by iterated refinements starting with an arbitrary initial partition. Hence, applying Theorem~\ref{Reg-lem} to $H$ with an initial partition $\{X,Y\}$, we obtain a vertex partition $\mathcal{P}^{(1)}= \{V_1, V_2, \dots, V_{t_1+t_2}\} $ of $V(H)$ where $1/\eta \le t_1+t_2\le t$ and a $2$-edge partition $\mathcal{P}^{(2)}=\cup_{1\le i<j\le t_1+t_2} \mathcal{P}^{ij}$. In particular, $H$ is $(\delta_3,r)$-regular w.r.t. $\mathcal{P}(2)=\{\mathcal{P}^{(1)}, \mathcal{P}^{(2)}\}$. Without loss of generality, we may assume that $X=\bigcup _{i=1}^{t_1}V_i$, $Y=\bigcup _{i=t_1+1}^{t_2}V_i$ and $t!\mid n$ (by discarding up to $t!$ vertices if necessary). 

Let $\mathcal{A}$ be the reduced hypergraph of $H$. Then $V(\mathcal{A})=\cup_{1\le i<j\le t_1+t_2} \mathcal{P}^{ij}$.
We below show that $\mathcal{A}$ is $p^*$-dense for some $p^*>\frac{1}{8}$, i.e. given a triple $\{i,j,k\}\in [t_1+t_2]^3$, 
\[
|E(\mathcal{A}^{ijk})|\geq p^*|\mathcal{P}^{ij}||\mathcal{P}^{ik}||\mathcal{P}^{jk}|
\]
holds for some $p^*>\frac{1}{8}$. As $H$ is an $(n,\frac{1}{8}+\varepsilon,\mu,\points)$ 3-graph, the number of $d_3$-useful edges in $E(V_i,V_j,V_k)$ is at least
\[
(\frac{1}{8}+\varepsilon )|V_i||V_j||V_k|-\mu n^3-2d_3n^3\geq (\frac{1}{8}+\frac{\varepsilon}{2} )|V_i||V_j||V_k|
\]
by our choices of constants and Lemma~\ref{useful-edge}.
On the other hand, by the triangle counting lemma, each triad $P_{abc}^{ijk}$ satisfies
\[
\mathcal{K}_3(P_{abc}^{ijk})\leq (d_2^3+3\delta)|V_i||V_j||V_k|,
\]
so we have
\[
(\frac{1}{8}+\frac{\varepsilon}{2} )|V_i||V_j||V_k|\leq |E(\mathcal{A}^{ijk})|\cdot (d_2^3+3\delta)|V_i||V_j||V_k|.
\]
Hence, by our choices of constants,
$|E(\mathcal{A}^{ijk})|\geq p^*|\mathcal{P}^{ij}||\mathcal{P}^{ik}||\mathcal{P}^{jk}|$ for some $p^*>\frac{1}{8}$ holds.

We begin to show that $(1,-1)\in L_{\mathcal{P},F}^{\lambda }(H)$ by contradiction. Choose arbitrary clusters $V_i,V_j,V_k\subseteq X$. As $H$ is an $(n,\frac{1}{8}+\varepsilon,\mu ,\points)$ 3-graph, we have $e_H(V_i,V_j,V_k)\geq (\frac{1}{8}+\varepsilon)|V_i||V_j||V_k|-\mu n^3$. Therefore, there are at least $(\frac{1}{8}+\varepsilon)|V_i||V_j||V_k|-\mu n^3-2d_3n^3>0$ $d_3$-useful edges in $E_H(V_i,V_j,V_k)$, which implies there is a tripartite graph $P^{ijk}_{abc}$ such that $H$ is $(d,\delta_3,r)$-regular w.r.t. $P^{ijk}_{abc}$ for some $d\geq d_3$. This gives us a $(3,3)$-complex by adding $E(H)\cap \mathcal{K}(P^{ijk}_{abc})$ as the ``3rd level". Here we turn $F$ into a $(3,3)$-complex by making each edge into a complete $i$-graph $K_3^{(i)}$ for $i=1,2$. By Lemma~\ref{Counting-lem}, there are at least $\frac{1}{2}(\frac{n}{t_1+t_2})^fd_3^{e(F)}d_2^{|\partial F|}\geq \lambda n^f$ copies of $F$ in $H[V_i\cup V_j\cup V_k]$. Thus, we have $(f,0)\in L_{\mathcal{P},F}^{\lambda }(H)$. Similarly, we have $(0,f),(f_1+f_2,f_3)\in L_{\mathcal{P},F}^{\lambda }(H)$. We assume that $(1,-1)\notin L_{\mathcal{P},F}^{\lambda }(H)$, so
\begin{gather}
(f-1,1),(1,f-1),(f_1+f_2-1,f_3+1)\notin L_{\mathcal{P},F}^{\lambda }(H). \tag{$\star$} 
\end{gather}

For any vertex class $\mathcal{P}^{i_1i_2}\subseteq V(\mathcal{A})$ and any vertex $P^{i_1i_2}_a\in \mathcal{P}^{i_1i_2}$, we color it as follows.

~\\
\noindent \textbf{Case 1.} If $i_1,i_2\in [t_1]$ or $i_1,i_2\in [t_1+t_2]\setminus[t_1]$, then $P^{i_1i_2}_a$ is colored by {\color{red}red} if there is an edge $\{P^{i_1i_2}_a,P^{i_1i_3}_b,P^{i_2i_3}_c\}\in E(\mathcal{A})$ for some $i_3\in [t_1+t_2]\setminus[t_1]$ while $P^{i_1i_2}_a$ is colored by {\color{blue}blue} if there is an edge $\{P^{i_1i_2}_a,P^{i_1i_3}_b,P^{i_2i_3}_c\}\in E(\mathcal{A})$ for some $i_3\in [t_1]$.

\noindent \textbf{Case 2.} If $i_1\in [t_1]$ and $i_2\in [t_1+t_2]\setminus[t_1]$, then $P^{i_1i_2}_a$ is colored by {\color{red}red} if there is an edge $\{P^{i_1i_2}_a,P^{i_1i_3}_b,P^{i_2i_3}_c\}\in E(\mathcal{A})$ for some $i_3\in [t_1]$ while $P^{i_1i_2}_a$ is colored by {\color{blue}blue} if there is an edge $\{P^{i_1i_2}_a,P^{i_1i_3}_b,P^{i_2i_3}_c\}\in E(\mathcal{A})$ for some $i_3\in [t_1+t_2]\setminus[t_1]$.
~\\

We claim that any vertex $P^{i_1i_2}_a\in V(\mathcal{A})$ can not receive two colors for $1\leq i_1<i_2\leq t_1+t_2$. For $i_1,i_2\in [t_1]$, if not, there exist two indices $i_3\in [t_1]$ and $i'_3\in [t_1+t_2]\setminus[t_1]$ such that both $\{P^{i_1i_2}_a,P^{i_1i_3}_b,P^{i_2i_3}_c\}$ and $\{P^{i_1i_2}_a,P^{i_1i'_3}_b,P^{i_2i'_3}_c\}$ form edges. By Lemma~\ref{Counting-lem}, for any fixed $v^*\in X_3$, there are at least $\frac{1}{2}(\frac{n}{t_1+t_2})^fd_3^{e(F)}d_2^{|\partial F|}\geq \lambda n^f$ copies of $F$ satisfying that $X_1$ is embedded into $V_{i_1}$, $X_2$ is embedded into $V_{i_2}$, $X_3\setminus \{v^*\}$ into $V_{i_3}$ and $\{v^*\}$ into $V_{i'_3}$. Therefore $(f-1,1)\in L_{\mathcal{P},F}^{\lambda }(H)$, which contradicts our assumption in ($\star$).  Similarly, for $i_1,i_2\in [t_1+t_2]\setminus[t_1]$, any vertex $P^{i_1i_2}_a\in V(\mathcal{A})$ can not receive two colors either. For Case 2, if not, there exist two indices $i_3\in [t_1]$ and $i'_3\in [t_1+t_2]\setminus[t_1]$ satisfying both $\{P^{i_1i_2}_a,P^{i_1i_3}_b,P^{i_2i_3}_c\}$ and $\{P^{i_1i_2}_a,P^{i_1i'_3}_b,P^{i_2i'_3}_c\}$ form edges. By Lemma~\ref{Counting-lem}, for any fixed $v^*\in X_2$, there are at least $\frac{1}{2}(\frac{n}{t_1+t_2})^f d_3^{e(F)}d_2^{|\partial F|}\geq \lambda n^f$ copies of $F$ satisfying that $X_1$ is embedded into $V_{i_1}$, $X_3$ is embedded into $V_{i_2}$, $X_2\setminus \{v^*\}$ into $V_{i_3}$ and $\{v^*\}$ into $V_{i'_3}$. Therefore $(f_1+f_2-1,f_3+1)\in L_{\mathcal{P},F}^{\lambda }(H)$, a contradiction.

Given any vertex class $\mathcal{P}^{i_1i_2}\subseteq V(\mathcal{A})$, let $\mathcal{P}^{i_1i_2}_{{\rm {\color{red}red}}}\subseteq \mathcal{P}^{i_1i_2}$ be the set of vertices colored by {\color{red}red} in $\mathcal{P}^{i_1i_2}$ and $\mathcal{P}^{i_1i_2}_{{\rm {\color{blue}blue}}}\subseteq \mathcal{P}^{i_1i_2}$ be the set of vertices colored by {\color{blue}blue} in $\mathcal{P}^{i_1i_2}$. Note that the following holds.

\begin{itemize}

 \item [{\rm (i)}] For any $i_1,i_2\in [t_1]$ and $i_3\in [t_1+t_2]\setminus[t_1]$, $E(\mathcal{A}^{i_1i_2i_3})=E_{\mathcal{A}}(\mathcal{P}^{i_1i_2}_{{\rm {\color{red}red}}},\mathcal{P}^{i_1i_3}_{{\rm {\color{red}red}}},\mathcal{P}^{i_2i_3}_{{\rm {\color{red}red}}})$.

 \item [{\rm (ii)}] For any $i_1\in [t_1]$ and $i_2,i_3\in [t_1+t_2]\setminus[t_1]$, $E(\mathcal{A}^{i_1i_2i_3})=E_{\mathcal{A}}(\mathcal{P}^{i_1i_2}_{{\rm {\color{blue}blue}}},\mathcal{P}^{i_1i_3}_{{\rm {\color{blue}blue}}},\mathcal{P}^{i_2i_3}_{{\rm {\color{blue}blue}}})$.

 \item [{\rm (iii)}] For any $i_1,i_2,i_3\in [t_1]$, $E(\mathcal{A}^{i_1i_2i_3})=E_{\mathcal{A}}(\mathcal{P}^{i_1i_2}_{{\rm {\color{blue}blue}}},\mathcal{P}^{i_1i_3}_{{\rm {\color{blue}blue}}},\mathcal{P}^{i_2i_3}_{{\rm {\color{blue}blue}}})$.

 \item [{\rm (iv)}] For any $i_1,i_2,i_3\in [t_1+t_2]\setminus[t_1]$, $E(\mathcal{A}^{i_1i_2i_3})=E_{\mathcal{A}}(\mathcal{P}^{i_1i_2}_{{\rm {\color{red}red}}},\mathcal{P}^{i_1i_3}_{{\rm {\color{red}red}}},\mathcal{P}^{i_2i_3}_{{\rm {\color{red}red}}})$.

\end{itemize}

First, by {\rm(iii)}, since $\mathcal{A}$ is $p^*$-dense with $p^* >1/8$, there must exist $i'_1,i'_2\in [t_1]$ such that $|\mathcal{P}^{i'_1i'_2}_{{\rm {\color{red}red}}}|<\frac{1}{2}|\mathcal{P}^{i'_1i'_2}|$. By {\rm (i)}, we have that for any $i_3\in [t_1+t_2]\setminus[t_1]$, $|\mathcal{P}^{i'_1i_3}_{{\rm {\color{red}red}}}|>\frac{1}{2}|\mathcal{P}^{i'_1i_3}|$ or $|\mathcal{P}^{i'_2i_3}_{{\rm {\color{red}red}}}|>\frac{1}{2}|\mathcal{P}^{i'_2i_3}|$ as $\mathcal{A}$ is $p^*$-dense. Let $I_1\subseteq [t_1+t_2]\setminus[t_1]$ be the set of indices satisfying the former inequality and $I_2\subseteq [t_1+t_2]\setminus[t_1]$ be the set of indices satisfying the latter. Without loss of generality, we assume $|I_1|\geq \frac{1}{2}t_2$. For any two indices $i_3,i'_3\in I_1$, we have $|\mathcal{P}^{i'_1i_3}_{{\rm {\color{blue}blue}}}|<\frac{1}{2}|\mathcal{P}^{i'_1i_3}|$ and $|\mathcal{P}^{i'_1i'_3}_{{\rm {\color{blue}blue}}}|<\frac{1}{2}|\mathcal{P}^{i'_1i'_3}|$ by the definition of $I_1$. 
Therefore, $|\mathcal{P}^{i_3i'_3}_{{\rm {\color{blue}blue}}}|>\frac{1}{2}|\mathcal{P}^{i_3i'_3}|$ holds by {\rm (ii)} and $p^*$-denseness of $\mathcal{A}$. This implies that for any two indices $i_3,i'_3\in I_1$, $|\mathcal{P}^{i_3i'_3}_{{\rm {\color{red}red}}}|<\frac{1}{2}|\mathcal{P}^{i_3i'_3}|$, which contradicts that $\mathcal{A}$ is $p^*$-dense  by {\rm (iv)}. 
\end{proof}

\begin{proof}[Proof of Theorem~\ref{1/8-dot-codegree}]
Suppose that $1/n\ll \mu\ll \lambda' \ll \zeta,  \varepsilon,\alpha, 1/f <1$ and $f, n \in \mathbb{N}$ with $n\in f\mathbb{N}$.
Let $F$ be a $3$-partite $3$-graph $F$ with $f:=v(F)$. By a result of Erd\H{o}s~\cite{E-1964}, $\pi_{\points}(F)=0$.
Suppose that $H$ is an $(n,\frac{1}{8}+\varepsilon,\mu,\points)$ $3$-graph with $\delta_2(H)\ge \alpha n$.
For each induced subhypergraph $H'$ on $V'\subseteq V(H)$ with $|V'|\ge 2\zeta n$, $H'$ is a $(v(H'), \frac{1}{8}+\varepsilon, \mu', \points)$ $3$-graph with $\mu'=\mu/(2\zeta)^3$. For $H'$ and a partition $\mathcal{P'}=\{X, Y\}$ of $V(H')$ with $|X|, |Y|\ge \zeta n$, we apply Lemma~\ref{transferral} to get $\lambda'$, then the property ($\heartsuit$) in Corollary~\ref{covertrans} holds for $H$.
By Proposition~\ref{cover1}, we can apply Corollary~\ref{covertrans} to find an $F$-factor of $H$. 
\end{proof}

\section{Proof of Theorem~\ref{factor3-2}}\label{3-partite3-pf}
In this section, we shall prove Theorem~\ref{factor3-2} by Corollary~\ref{covertrans} and the regularity method. 
For a $3$-graph $F\in \cover_{3,>}^{\points,2}\cap \trans_{3}^2$, our aim is to prove $F\in \factor_3^{\points,2}$. Similar to the proof of Theorem~\ref{1/8-dot-codegree},
 we only need to prove
 a result (see Lemma~\ref{transferral3-2} ) similar to Lemma~\ref{transferral} and check $\pi_{\points}(F)=0$ for $F\in \cover_{3,>}^{\points,2}$.
 In 2018, Reiher--R\"odl--Schacht~\cite{Reiher2018Hypergraphs} gave a characterization of all 3-graphs $F$ with  $\pi_{\points}(F)=0$. In particular, they constructed an $(n, \frac{1}{27}, \mu, \points)$ 3-graphs $H$ with $\delta_2(H)\ge (\frac{1}{27}-\mu)n$ such that  $H$ does not contain any $F$ with $\pi_{\points}(F)>0$ from random colorings
 of $[n]^2$. By this construction, we have the following proposition. 
 
 \begin{prop}\label{cover-density}
If 3-graph $F\in \cover_{3,>}^{\points,2}$, then $\pi_{\points}(F)=0$.
 \end{prop}

\begin{proof}
Let $F\in \cover_{3,>}^{\points,2}$.  Suppose that $\pi_{\points}(F)>0$.
By the construction of Reiher--R\"odl--Schacht~\cite{Reiher2018Hypergraphs}, there exists an $F$-free  $(n, \frac{1}{27}, \mu, \points)$ 3-graph $H$ with $\delta_2(H)\ge (\frac{1}{27}-\mu)n$, which contradicts $F\in \cover_{3,>}^{\points,2}$.
\end{proof}





\begin{lemma}\label{transferral3-2}
Suppose that $1/n\ll \mu \ll \lambda \ll p,\zeta<1$ and $n\in \mathbb{N}$. Let $F\in \trans_3^2$ and $\pi_{\points}(F)=0$.
If $H$ is an $(n,p,\mu,\points)$
$3$-graph with vertex partition $\mathcal{P}=\{X,Y\}$ and $|X|,|Y|\geq \zeta n$, then $(-1,1)\in L_{\mathcal{P},F}^{\lambda }(H)$.
\end{lemma}

We shall use the following lemmas to prove Lemma~\ref{transferral3-2}.
Recall that for a $k$-graph $F$, $\partial F:=\{S\in [V(F)]^{k-1}: \exists\ e\in E(F), S\subseteq e\}$.


In our previous work~\cite{DHSWZ}, based on the work of Reiher-R\"odl-Schacht~\cite{Reiher2018Hypergraphs}, we show that a $3$-graph $F$ with $\pi_{\points}(F)=0$ admits a structural $3$-coloring of $\partial F$ w.r.t. a 2-shadow disjoint partition $\mathcal{P}$.

\begin{lemma}{\rm (\cite[Lemma 4.2]{DHSWZ}).}
\label{alter} For a $3$-graph $F$ with $\pi_{\points}(F)=0$ and a 2-shadow disjoint partition $\mathcal{P}=\{A,B\}$ of $V(F)$, there is an enumeration $\{v_1,v_2,\dots,v_f\}$ of $V(F)$ with $\{v_1,\cdots,v_{|A|}\}=A$ and $\{v_{|A|+1},\cdots,v_f\}=B$ and there is a $3$-coloring $\varphi:\partial F \longmapsto \{{\color {red}red},~{\color{blue}blue},~{\color{green}green}\}$ such that every hyperedge $\{v_i,v_j,v_k\}\in E(F)$ with $i<j<k$ satisfies: 
\[
\varphi(v_i,v_j)={\color {red}red},\  \ \varphi(v_i,v_k)={\color{blue}blue},\ \ \varphi(v_j,v_k)={\color{green}green}.
\]

\end{lemma}

\begin{lemma}\label{Rembed}{\rm (\cite[Lemma 3.1]{Reiher2018Hypergraphs}).}
Given $\eta>0$ and $h\in \mathbb{N}$, there exists an integer $q$ such that the following holds. If $\mathcal{A}$ is an $\eta$-dense reduced hypergraph with index set $[q]$, vertex class $\mathcal{P}^{ij}$ and constituents $\mathcal{A}^{ijk}$, then
\begin{itemize}
  \item [{\rm (i)}] there are indices $\lambda(1)<\cdots<\lambda(h)$ in $[q]$ and
  \item [{\rm (ii)}] for each pair $1\leq r<s\leq h$ there are three vertices $P^{\lambda(r)\lambda(s)}_{{\rm{\color{red} red}}}$, $P^{\lambda(r)\lambda(s)}_{{\rm {\color{blue}blue}}}$ and $P^{\lambda(r)\lambda(s)}_{{\rm {\color{green}green}}}$ in $\mathcal{P}^{\lambda(r)\lambda(s)}$ such that for every triple indices $1\leq r<s<t\leq h$ the three vertices $P^{\lambda(r)\lambda(s)}_{{\rm{\color{red} red}}}$, $P^{\lambda(r)\lambda(t)}_{{\rm {\color{blue}blue}}}$ and $P^{\lambda(s)\lambda(t)}_{{\rm {\color{green}green}}}$ form a hyperedge in $\mathcal{A}^{\lambda(r)\lambda(s)\lambda(t)}$.
\end{itemize}
\end{lemma}

\begin{lemma}{\rm (\cite[Lemma 4.8]{DHSWZ}).}
\label{biembed}
Given $\eta>0$ and $h_1,h_2\in \mathbb{N}$, there exist integers $q_1$ and $q_2$ such that the following holds. If $\mathcal{A}$ is an $\eta$-dense reduced hypergraph with index set $[q_1+q_2]$, vertex class $\mathcal{P}^{ij}$ and constituents $\mathcal{A}^{ijk}$, then there are indices $\lambda(1)<\cdots<\lambda(h_1)$ in $[q_1]$ and $\sigma(1)<\cdots<\sigma(h_2)$ in $[q_1+q_2]\setminus [q_1]$ satisfying that
\begin{itemize}
  \item [{\rm (i)}] for each pair $1\leq r<s\leq h_1$ there is a vertex $P^{\lambda(r)\lambda(s)}_{{\rm {\color{red}red}}}$;
  \item [{\rm (ii)}] for any $1\leq r\leq h_1$ and any $1\leq t\leq h_2$ there are two vertices $P^{\lambda(r)\sigma(t)}_{{\rm {\color{blue}blue}}}$ and $P^{\lambda(r)\sigma(t)}_{{\rm {\color{green}green}}}$;
  \item [{\rm (iii)}] for every triple indices $1\leq r<s\leq h_1$ and $1\leq t\leq h_2$ the three vertices $P^{\lambda(r)\lambda(s)}_{{\rm {\color{red}red}}}$, $P^{\lambda(r)\sigma(t)}_{{\rm {\color{blue}blue}}}$ and $P^{\lambda(s)\sigma(t)}_{{\rm {\color{green}green}}}$ form a hyperedge in $\mathcal{A}^{\lambda(r)\lambda(s)\sigma(t)}$.
\end{itemize}
\end{lemma}

\noindent{\textbf{Remark.}{\rm (\cite{DHSWZ}).}} In the following proof we also need another statement of Lemma~ \ref{biembed} as follows.

\emph{
Given $\eta>0$ and $h_1,h_2\in \mathbb{N}$, there exist integers $q_1$ and $q_2$ such that the following holds. If $\mathcal{A}$ is an $\eta$-dense reduced hypergraph with index set $[q_1+q_2]$, vertex class $\mathcal{P}^{ij}$ and constituents $\mathcal{A}^{ijk}$, then there are indices $\lambda(1)<\cdots<\lambda(h_1)$ in $[q_1]$ and $\sigma(1)<\cdots<\sigma(h_2)$ in $[q_1+q_2]\setminus [q_1]$ satisfying that
\begin{itemize}
  \item [{\rm (i)}] for each pair $1\leq s<t\leq h_2$ there is a vertex $P^{\sigma (s)\sigma (t)}_{{\rm {\color{green}green}}}$;
  \item [{\rm (ii)}] for any $1\leq r\leq h_1$ and any $1\leq t\leq h_2$ there are two vertices $P^{\lambda(r)\sigma(t)}_{{\rm {\color{red}red}}}$ and $P^{\lambda(r)\sigma(t)}_{{\rm {\color{blue}blue}}}$;
  \item [{\rm (iii)}] for every triple indices $1\leq r\leq h_1$ and $1\leq s<t\leq h_2$ the three vertices $P^{\lambda(r)\sigma (s)}_{{\rm {\color{red}red}}}$, $P^{\lambda(r)\sigma(t)}_{{\rm {\color{blue}blue}}}$ and $P^{\sigma (s)\sigma(t)}_{{\rm {\color{green}green}}}$ form a hyperedge in $\mathcal{A}^{\lambda(r)\sigma (s)\sigma(t)}$.
\end{itemize}
}

\begin{proof}[Proof of Lemma~\ref{transferral3-2}]
Given  $F\in \trans_3^2$ with $\pi_{\points}(F)=0$ and $f:=v(F)$, let $\mathcal{F}_1=\{A_1,B_1\},\mathcal{F}_2=\{A_2,B_2\},\cdots,\mathcal{F}_m=\{A_{m},B_m\}$ be all 2-shadow disjoint bipartions of $V(F)$ with $|A_i|=a_i$ and $|B_i|=b_i$ for $i\in[m]$.  Choose constants satisfying the following hierarchy:
\[
1/n\ll \lambda,\mu \ll 1/r,\delta\ll d_2, \delta_3,1/t \ll 1/M_1 \ll 1/M_2 \ll 1/m_2\ll 1/m_1, d_3\ll p,\zeta,1/f
\]
where $d_2\geq 1/t$ and $1/d_2\in \mathbb{N}$. 
As in the proof of Lemma~\ref{transferral}, 
we apply Theorem~\ref{Reg-lem} to $H$ with the initial partition $\{X,Y\}$ and obtain a vertex partition $\mathcal{P}^{(1)}= \{V_1, V_2, \dots, V_{t_1+t_2}\} $ of $V(H)$ with $1/\eta \le t_1+t_2\le t$ and a $2$-edge partition $\mathcal{P}^{(2)}=\cup_{1\le i<j\le t_1+t_2} \mathcal{P}^{ij}$. In addition, $H$ is $(\delta_3,r)$-regular w.r.t. $\mathcal{P}(2)=\{\mathcal{P}^{(1)}, \mathcal{P}^{(2)}\}$. Without loss of generality, we may assume that $X=\bigcup _{i=1}^{t_1}V_i$, $Y=\bigcup _{i=t_1+1}^{t_2}V_i$ and $t!\mid n$ (by discarding up to $t!$ vertices if necessary).
Let $\mathcal{A}$ be the reduced hypergraph of $H$. Then $V(\mathcal{A})=\cup_{1\le i<j\le t_1+t_2} \mathcal{P}^{ij}$.
Similar to the proof of Lemma~\ref{transferral},  $\mathcal{A}$ is $p^*$-dense for some $p^*>0$.

 For any $\mathcal{F}_i=\{A_i,B_i\}$ where $i\in [m]$, we first prove that $(a_i,b_i)\in I_{\mathcal{P},F}^{\lambda}(H)$. For any integer $0\leq j\leq 3$, let $F^{\mathcal{F}_i}_j\subseteq F$ be the 3-graph consisting of all 3-edges $e$ with $|e\cap A_i|=j$. An easy observation is that $\partial F^{\mathcal{F}_i}_0$, $\partial F^{\mathcal{F}_i}_1$, $\partial F^{\mathcal{F}_i}_2$, $\partial F^{\mathcal{F}_i}_3$ are pairwise disjoint as $\mathcal{F}_i$ is a 2-shadow disjoint bipartition.
In order to use Lemma~\ref{Counting-lem}, we need to find a regular $f$-partite graph. 
We first apply Lemma~\ref{Rembed} with index set $[t_1]$ to find an index subset $\mathcal{I}_1$ of $[t_1]$ with $|\mathcal{I}_1|=M_1$ such that Property {\rm{(ii)}} from Lemma~\ref{Rembed} is satisfied. Similarly, we apply Lemma~\ref{Rembed} with index set $[t_1+t_2]\setminus [t_1]$ to find an index subset $\mathcal{I}_2$ of cardinality $M_2$ in $[t_1+t_2]\setminus [t_1]$. Next we apply Lemma \ref{biembed} with $h_1:=m_1$, $h_2:=m_2$ and the set $\mathcal{I}_1\cup\mathcal{I}_2$ to find a subset $\mathcal{J}_1$ of $\mathcal{I}_1$ with cardinality $m_1$ and a subset $\mathcal{J}_2$ of $\mathcal{I}_2$ with cardinality $m_2$ such that Properties {\rm(i)}--{\rm(iii)} from Lemma \ref{biembed} are satisfied.
Finally, we apply the Remark of Lemma \ref{biembed} with $h_1:=a_i$, $h_2:=b_i$ and the index set $\mathcal{J}_1\cup\mathcal{J}_2$ to find a subset $\mathcal{X}$ of $\mathcal{J}_1$ of cardinality $a_i$ and a subset $\mathcal{Y}$ of $\mathcal{J}_2$ of cardinality $b_i$ satisfying Properties {\rm(i)}--{\rm(iii)} from Lemma \ref{biembed}. Without loss of generality, we may assume that $\mathcal{X}=\{\lambda(1),\dots,\lambda(a_i)\}$ and $\mathcal{Y}=\{\lambda(a_i+1),\dots,\lambda(f)\}$ with $\lambda(1)<\dots<\lambda(f)$.

By Lemma~\ref{alter}, there is an enumeration of $V(F)$ and a $3$-coloring w.r.t. $\mathcal{F}_i$ satisfying the property in Lemma~\ref{alter}. For any $\{v_i,v_j,v_k\}\in \bigcup _{j=0}^3F^{\mathcal{F}_i}_j$, by our choices of $\mathcal{X}$ and $\mathcal{Y}$, we have $d(H|P^{\lambda(i)\lambda(j)\lambda(k)}_{{\rm {\color{red}red},{\color{blue}blue},{\color{green}green}}})\geq d_3$. Since $\partial F^{\mathcal{F}_i}_0$, $\partial F^{\mathcal{F}_i}_1$, $\partial F^{\mathcal{F}_i}_2$, $\partial F^{\mathcal{F}_i}_3$ are pairwise disjoint, by Lemma~\ref{Counting-lem}, there are at least $\frac{1}{2}(\frac{n}{t_1+t_2})^kd_3^{e(F)}d_2^{|\partial F|}\geq \lambda n^k$ copies of $F$ satisfying that $v_i$ is embedded into $V_{\lambda (i)}$ for $1\leq i\leq f$. Therefore $(a_i,b_i)\in I_{\mathcal{P},F}^{\lambda }(H)$.

As $F\in \trans_3^2$, i.e. $(1,-1)\in L_F^2$, we have $(1,-1)\in L_{\mathcal{P},F}^{\lambda }(H)$.
\end{proof}

\begin{proof}[Proof of Theorem~\ref{factor3-2}]
Consider $F\in \cover_{3, >}^{\points,2}\cap \trans^2_3$ with $f:=v(F)$. We choose constants satisfying the following hierarchy:
\[
1/n \ll\mu \ll \lambda' \ll \zeta\ll  p,\alpha,1/f<1
\]
with $n\in f\mathbb{N}$. 
Let $H$ be an $(n,p,\mu,\points)$ $3$-graph with $\delta_2(H)\ge \alpha n$. 
For each induced subhypergraph $H'$ on $V'\subseteq V(H)$ with $|V'|\ge 2\zeta n$, $H'$ is a $(v(H'), p, \mu', \points)$ $3$-graph with $\mu'=\mu/(2\zeta)^3$. Note that  we have $\pi_{\points}(F)=0$ since $F\in \cover_{3, >}^{\points,2}$. For $H'$ and a partition $\mathcal{P'}=\{X, Y\}$ of $V(H')$ with $|X|, |Y|\ge \zeta n$, we apply Lemma~\ref{transferral3-2} to get $\lambda'$, then the property ($\heartsuit$) in Corollary~\ref{covertrans} holds for $H$.
By Corollary~\ref{covertrans}, we can find an $F$-factor of $H$.
\end{proof}

\begin{proof}[Proof of Theorem~\ref{3-partite3-graph}]

Let $F$ be a $3$-partite $3$-graph with  $f$ vertices. By Theorem~\ref{1/8-dot-codegree}, we have $p_F\le 1/8$. Suppose that $F\in \trans_3^2$.
We apply Proposition~\ref{cover1} with $k=3$ and obtain that  $F\in \cover_{3,>}^{\points ,2}$. By Theorem~\ref{factor3-2}, we have $F\in\factor_3^{\points,2}$ which implies that $p_F=0$.

Next, let $F\notin \trans_3^2$. Then it suffices to show that $p_F\ge 1/8$.
Assume by contradiction that $p_F< 1/8$ which implies that 
for every $0<\alpha <1$, there exists an $n_0$ and $\mu>0$ such that  every $(n, 1/8,\mu ,\points)$ 3-graph $H$ with $\delta_2(H)\geq \alpha n$, $n\ge n_0$ and $n\in f\mathbb{N}$,  satisfies that $H$ has an $F$-factor. 
However, we will prove $F \in \trans^{2}_3$ by the following construction, which contradicts  $F\notin \trans_3^2$.
For $n\in f\mathbb{N}$, define a probability distribution $H(n)$ on 3-graphs of order $n$ as follows.
Let $K$ be the complete graph on $n$ vertices, and $\mathcal{P}=\{X, Y\}$ be a partition of $V(K)$ satisfying $\textbf{i} _{\mathcal{P}}(V(K))=(n_1,n_2)$ with $n_1,n_2\ge n/3$.
Define $\phi:E(K) \longmapsto \{red, blue\}$ as a random $2$-coloring with each color associated to an edge with probability $\frac{1}{2}$ independently.
Let now the vertex set of $H(n)$ be $V(K)$ and include a set $e\in [V(K)]^{3}$ in $E(H(n))$ as follows. If  $|e\cap X|$ is even, then make $e\in E(H(n))$ if $K[e]$
  is a  red triangle in $K$.
 If  $|e\cap X|$ is odd, then make $e\in E(H(n))$ if $K[e]$
 is a  blue triangle in $K$.
By concentration inequalities (e.g. the Janson inequality) and the union bound, for sufficiently large $n$,
there exists  $H\in H(n)$ such that $H$ is an $(n, 1/8, \mu ,\points)$  3-graph with $\delta _2(H)=\Omega (n)$. Then $H$ has an $F$-factor.
Moreover, $\mathcal{P}$ is a $2$-shadow disjoint bipartition of $V(H)$ since for any two edges $e, e'\in E(H)$ with 
$\textbf{i} _{\mathcal{P},H}(e)\neq \textbf{i} _{\mathcal{P}, H}(e')$,
either $K[e]$ and $K[e']$ have different colors, or $|e\cap e'|<2$, 
each of which implies that $|e\cap e'|<2$.
For every copy $F'$ of $F$ in $H$, as $F'$ is a subgraph of $H$, the bipartition $\mathcal{P}'$ of $V(F')$ inherited from the bipartition $\mathcal{P}$ is also $2$-shadow disjoint, implying that $\textbf{i}_{\mathcal P}(V(F'))=\textbf{i}_{\mathcal P', F'}\in L_F^2$.
Summing up the $\textbf{i}_{\mathcal P}(V(F'))$ over all $F'\in \mathcal F$ gives that  $(n_1,n_2)\in L^2_F$.
Analogously we can define a new 3-graph $H'$ satisfying $(n,1/8,\mu ,\points)$-denseness and $\delta _2(H')=\Omega (n)$,
which has a $2$-shadow disjoint bipartition $\mathcal{P}^*$ with $\textbf{i} _{\mathcal{P^*}, H'}=(n_1-1,n_2+1)$.
The same argument shows that $(n_1-1,n_2+1)\in L^2_F$.
Therefore we obtain $(1,-1)\in L^2_F$, that is, $F\in \trans^2_3$.
\end{proof}

\section{The proof of Theorem~\ref{gcd}}\label{k-partitek-pf}

In this section, we shall prove Theorem~\ref{gcd}. Similar to the proof of Theorem~\ref{1/8-dot-codegree}, we only need to prove a lemma similar to Lemma~\ref{transferral}.
First, we need the following proposition from~\cite{Han2017Minimum}.

\begin{prop}{\rm{(\cite[\rm Proposition 3.1]{Han2017Minimum})}.}
\label{supersaturation}
Given integers $k,r_0,a_1,\dots ,a_k\in \mathbb{N}$, suppose that $1/n\ll \lambda \ll \eta ,1/k,1/r_0,1/a_1,\dots ,1/a_k$. Let $H$ be a $k$-graph on $n$ vertices with a vertex partition $V_1\cup \cdots \cup V_r$ where $r\leq r_0$. Suppose $i_1,\dots ,i_k\in [r]$ and $H$ contains at least $\eta n^k$ edges $e=\{v_1,\dots ,v_k\}$ such that $v_1\in V_{i_1},\dots ,v_k\in V_{i_k}$. Then $H$ contains at least $\lambda n^{a_1+\cdots +a_k}$ copies of $K_{a_1,\dots ,a_k}$ whose jth part is contained in $V_{i_j}$ for all $j\in[k]$.
\end{prop}

\begin{lemma}\label{k-transferral}
Suppose that $1/n \ll \mu \ll \lambda \ll p,\zeta<1$ and $n\in \mathbb{N}$. Let $F$ be a $k$-partite $k$-graph with $\gcd(\mathcal S(F))=1$ and $f:=v(F)$.
If $H$ is an $(n,p,\mu,\points)$ $k$-graph with vertex partition $\mathcal{P}=\{X,Y\}$ of $V(H)$ and $|X|,|Y|\ge \zeta n$, then $(1,-1)\in L_{\mathcal{P},F}^{\lambda }(H)$.
\end{lemma}

\begin{proof}
Suppose that $F$ is a $k$-partite $k$-graph of order $f$ satisfying $\gcd(\mathcal S(F))=1$.
Let $K_1,\dots, K_t$ be all $k$-partite realisations of $F$ with $t\ge 1$. For any $i\in [t]$, let $V_{i1}, \dots, V_{ik}$ denote the vertex classes of  $K_i$ with $|V_{ij}|= a_{ij}$
for $j\in  [k]$. Choose constants satisfying the following hierarchy:
\[
1/n\ll \mu \ll \lambda  \ll \eta \ll p,\zeta,1/k,1/f.
\]
As $H$ is $(p,\mu,\points)$-dense, we have $e(X,\dots ,X)\geq p|X|^k-\mu n^k\geq p(\zeta n)^k-\mu n^k$. Then there are at least $\frac{1}{k!}e(X,\dots ,X)\geq \eta n^k$ edges in $H[X]$. By Proposition~\ref{supersaturation}, there are at least $\lambda n^f$ copies of $K_i$ in $H[X]$ for each $i\in [t]$, which implies that $(f,0)\in L_{\mathcal{P},K_i}^{\lambda }(H)$. Since $e(Y,X,\dots ,X)\geq p|Y||X|^{k-1}-\mu n^k\geq p(\zeta n)^k-\mu n^k$, by Proposition~\ref{supersaturation} there are at least $\lambda n^f$ copies of $K_i$ where $V_{i1}$ is contained in $Y$ and $V_{ij}$ is contained in $X$ for $j\in [k]\setminus \{1\}$. This implies that $(\sum_{j\in [k]\setminus \{1\}}a_{ij}, a_{i1})\in L_{\mathcal{P},K_i}^{\lambda }(H)$. Similarly, $(\sum_{j\in [k]\setminus \{2\}}a_{ij}, a_{i2}),\dots ,(\sum_{j\in [k]\setminus \{k\}}a_{ij}, a_{ik})\in L_{\mathcal{P},K_i}^{\lambda }(H)$. Since $K_i$ is a $k$-partite realisation of $F$, $L_{\mathcal{P},K_i}^{\lambda }(H)=L_{\mathcal{P},F}^{\lambda }(H)$ for all $i\in [t]$. Therefore, $(a_{ij},-a_{ij})\in L_{\mathcal{P},F}^{\lambda }(H)$ for all $i\in [t]$ and $j\in [k]$. As $\gcd(\mathcal S(F))=1$, there exist integers $m_{ij}$ for $i\in [t]$ and $j\in [k]$ such that $\sum_{i,j}m_{ij}a_{ij}=1$.
Hence $(1,-1)=\sum_{i,j}m_{ij}(a_{ij},-a_{ij})\in L_{\mathcal{P},F}^{\lambda }(H)$.
\end{proof}

\begin{proof}[Proof of Theorem~\ref{gcd}]
Suppose that $1/n\ll \mu\ll \lambda' \ll \zeta, p,\alpha, 1/f <1$ and $f, n \in \mathbb{N}$ with $n\in f\mathbb{N}$.
Let $F$ be a $k$-partite $k$-graph $F$ satisfying $\gcd(\mathcal S(F))=1$ and $f:=v(F)$. By the result of Erd\H{o}s~\cite{E-1964}, $\pi_{\points}(F)=0$.
Suppose that $H$ is an $(n,p,\mu,\points)$ $k$-graph with $\delta_{k-1}(H)\ge \alpha n$.
For each induced subhypergraph $H'$ on $V'\subseteq V(H)$ with $|V'|\ge 2\zeta n$, $H'$ is a $(v(H'), p, \mu', \points)$ $k$-graph with $\mu'=\mu/(2\zeta)^k$. For $H'$ and a partition $\mathcal{P'}=\{X, Y\}$ of $V(H')$ with $|X|, |Y|\ge \zeta n$, we apply Lemma~\ref{k-transferral} to get $\lambda'$, then the property ($\heartsuit$) in Corollary~\ref{covertrans} holds for $H$.
By Proposition~\ref{cover1}, we can apply Corollary~\ref{covertrans} to find an $F$-factor of $H$.

In particular, if $F$ is a complete $k$-partite $k$-graph satisfying $\gcd(\mathcal S(F))>1$, then we claim that  $F\notin \trans^{k-1}_k$.
This together with Observation~\ref{construction-trans} would imply $F\notin \factor_k^{\points, k-1}$, which in turn implies the moreover part of the theorem.
Indeed, let $U_1, \dots, U_k$ be the vertex classes of $F$ and $\mathcal{P}=\{V_1, V_2\}$ be a $(k-1)$-shadow disjoint bipartition of $V(F)$. For any $i\in [k]$ and any $x,y \in U_i$ we may choose vertices $u_j\in U_j$ for each $j\neq i$, and since $F$ is complete both $\{x\}\cup \{u_j : j\neq i\}$ and $\{y\}\cup \{u_j : j\neq i\}$ are edges of $F$.
Both these sets therefore have the same index vector with respect to  $\mathcal{P}$, so we must have $x, y\in V_1$ or $x, y\in V_2$. It follows that $U_i \subseteq V_1$ or $U_i \subseteq V_2$ for each $i\in [k]$, and therefore that $\gcd(\mathcal S(F))$ divides $|V_t|$ for $t\in [2]$. Since $L^{k-1}_{F}$ is the lattice generated by all $\textbf{i} _{\mathcal{P},F}$ such that $\mathcal{P}$ is $(k-1)$-shadow disjoint,  $\gcd(\mathcal S(F))$ divides $a$ and $b$ for each  $(a,b)\in L^{k-1}_{F}$. Due to $\gcd(\mathcal S(F))>1$, we conclude that  $(1,-1)\notin L^{k-1}_{F}$, and so $F\notin \trans^{k-1}_k$.
\end{proof}

\section{Concluding Remarks}\label{remark}

In this paper, we studied the $F$-factor problems in quasi-random $k$-graphs.
We summarize our contributions as follows.
\begin{enumerate}[label=(\arabic*)]
\item For a 3-partite 3-graph $F$, we show that $p_F=0$ if $F\in \trans_3^{2}$, and $p_F=1/8$ otherwise (Theorem~\ref{3-partite3-graph}), which also implies that $1/8$ is the density threshold for ensuring all 3-partite 3-graphs $F$-factors  in $\points$-quasi-random 3-graphs given a  minimum codegree condition $\Omega(n)$.
We also show that $p$-dense $\points$-quasi-random 3-graphs plus a minimum vertex degree is not enough to ensure it contains  $F$-factors for all 3-partite 3-graphs $F$, even when $p$ is arbitrarily close to 1 and $F=K_{2,2,2}$ (Theorem~\ref{construction}).


\item For $k\geq 3$, using $\factor^{\edge,1}_k=\cover^{\edge,1}_k$ (Theorem~\ref{factor-cover2}), we show that all $k$-partite $k$-graphs are in $\factor^{\edge,1}_k$ (Theorem~\ref{dot-edge-partite}).

\item For $k=3$,
 we show that $\cover_{3,>}^{\points ,2}\cap \trans^2_3\subseteq \factor_3^{\points ,2} $ (Theorem~\ref{factor3-2}).
By Observation~\ref{construction-trans} and Theorem~\ref{3-partite3-graph}, we obtain that for each 3-partite 3-graph $F$,  $F\in \factor_3^{\points ,2}$ if and only if  $F\in \trans^2_3$. 

\item For $k\ge 3$, we find that $k$-partite $k$-graphs $F$ with $\gcd(\mathcal{S}(F))=1$ belong to $\factor_k^{\points , k-1}$ (Theorem~\ref{gcd}).
Moreover, the condition $\gcd(\mathcal{S}(F))=1$ is necessary for complete $k$-partite $k$-graphs $F$.

\end{enumerate}

Below we include some further problems and results for the following stronger quasi-randomness for $k$-graphs.

\begin{defi}
Given integers $n\geq k\ge 2$, let real numbers $p\in [0,1]$, $\mu >0$ be given and let $H=(V,E)$ be a $k$-graph with $n$ vertices. We say that $H$ is ($p,\mu,\cherry$)\emph{-dense} if for any two families of ordered $(k-1)$-tuples $P,Q\subseteq V^{k-1}$,
\begin{equation}
e_{\cherry}(P,Q)\geq p|\mathcal{K}_{\cherry}(P,Q)|-\mu n^k,
\end{equation}
where $\mathcal{K}_{\cherry}(P,Q)=\{(x_1,x_2,\dots,x_{k-1},x_k)\in V^k :(x_1,\dots,x_{k-1})\in P \text{\ and } (x_2,\dots,,x_k)\in Q \}$, and
$e_{\cherry}(P,Q)$ is the size of the set $\{(x_1,x_2,\dots,x_{k-1},x_k)\in \mathcal{K}_{\cherry}(P,Q): \{x_1,x_2,\dots,x_{k-1},x_k\}\in E \} $.
\end{defi}

Inspired by Problem~\ref{prob}, we raise a more general problem.

\begin{prob}\label{prob2}
For $k\ge 3$, $1\le s\le k-1$ and $\mathscr{S}\in \{\points, \edge, \cherry\}$, which $k$-graphs $F$ belong to $\factor_k^{\mathscr{S} ,s}$?
\end{prob}

Clearly, $\mathcal{C}(n,p,\mu, \cherry)\subseteq \mathcal{C}(n,p,\mu, \edge)\subseteq \mathcal{C}(n,p,\mu, \points)$ holds for all $p\in [0,1]$ and $\mu >0$. Similar to Theorem~\ref{factor-cover1} and Theorem~\ref{factor-cover2}, we have the following result.

\begin{thm}\label{cherry-equal}\label{general-thm}
For any $k \geq 2$ and $\mathscr{S}\in \{\points, \edge, \cherry\}$, $\factor^{\mathscr{S},1}_k=\cover^{\mathscr{S},1}_{k}$.
\end{thm}

In fact, the proof of Theorem~\ref{general-thm} uses a key property: each $k$-graph $F$ with $F\in \cover^{\mathscr{S},1}_{k}$ 
 also belongs to $\cover^{\mathscr{S},1}_{k,>}$.  
However, for $2\le s\le k-1$, we do not know whether this property still holds.
In addition, the ranges of $\factor^{\mathscr{S},s}_k$ and $\cover^{\mathscr{S},s}_{k}$ are larger than that of $\factor^{\mathscr{S},1}_k$ and $\cover^{\mathscr{S},1}_{k}$ respectively,
 which also makes it more difficult to characterize $\factor^{\mathscr{S},s}_k$ completely, even for $\factor^{\points,2}_3$. 
By Theorem~\ref{factor3-2}, if each 3-graph $F$ with $F\in \cover^{\points,2}_{k}$ 
satisfies $F\in \cover^{\points,2}_{k,>}$, then we have $\factor_3^{\points ,2} =\cover_{3}^{\points ,2}\cap \trans^2_3$. Therefore,  we raise the following problem.

\begin{prob}\label{prob3}
For $1<s<k$, is it true that $ \cover^{\points,s}_{k}= \cover^{\points,s}_{k,>}$?
\end{prob}

At last we include a result on $\cherry$-quasi-randomness, which is an application of Theorem~\ref{cherry-equal}.
Let $K_4^{(3)-}$ be the (unique) $3$-graph consisting of three hyperedges on four vertices. 
It is known that
\[
\pi_{\points}(K_4^{(3)-})=\pi_{\edge}(K_4^{(3)-})=1/4 \text{~~\cite{Glebov_2015, Reiher_2018}  \ and \ } \pi_{\cherry}(K_4^{(3)-})=0.
\]
Therefore, $K_4^{(3)-}\notin \factor_3^{\edge,1}$. 
By using Theorem~\ref{cherry-equal},  we show the following result.

\begin{thm}\label{K-4-3}
$K_4^{(3)-}\in \factor_3^{\cherry,1}$.
\end{thm}

\begin{proof}
Suppose that $n_0$ is large enough and $\mu$ is sufficiently small. By Theorem~\ref{cherry-equal}, it suffices to show that $K_4^{(3)-}\in \cover^{\cherry,1}_{3}$. To see this, let $H$ be an $n$-vertex $(p,\mu, \cherry)$-dense $3$-graph with $\delta_1(H)\ge \alpha n^2$, $n\in 4\mathbb{N}$ and $n\geq n_0$.
Fixed $w\in V(H)$,
construct an auxiliary graph $G$ with $V(G)=V(H)\setminus \{w\}$ and $E(G)=\big\{\{v_1,v_2\}:\{v_1,v_2\}\in N(w)\big\}$. We shall count pairs $(v, T)$ where $T$ is a pair of vertices such that both are adjacent to $v$ in $G$. By Jensen's inequality, the number of such pairs is
\[
\sum_{v\in V(G)}\binom{d_G(v)}{2}\ge (n-1)\binom {\frac{1}{n-1}\sum_{v\in V(G)}d_G(v)}{2}\ge (n-1)\binom {2\alpha n}{2}\ge \alpha n^3.
\]
Let $P=Q=N(w)$. Since $H$ is ($p,\mu,\cherry$)\emph{-dense}, we have
\[
e_{\cherry}(P,Q)\geq p|\mathcal{K}_{\cherry}(P,Q)|-\mu n^3\ge (2\alpha p-\mu) n^3>0,
\]
which implies there are two pairs $\{x, y\}, \{x, z\}\in N(w)$ such that $\{x,y,z\}\in E(H)$. Therefore, $\{w,x,y\}$, $\{w,x,z\}$ and $\{x,y,z\}$ form a copy of $K_4^{(3)-}$ in $H$.
\end{proof}

\bibliographystyle{abbrv}
\bibliography{ref,ref1}
\end{document}